\documentclass[preprint,12pt]{elsarticle}

\usepackage{amsmath, amssymb}
\usepackage{graphics}
\usepackage{caption}
\usepackage{subcaption}
\usepackage{amsthm}
\usepackage{color}
\usepackage{hyperref}
\usepackage{tikz}

\usepackage{url}

\def\checkmark{\tikz\fill[scale=0.4](0,.35) -- (.25,0) -- (1,.7) -- (.25,.15) -- cycle;}

\definecolor{orange}{rgb}{1,0.5,0}

\newtheorem{theorem}{{\bf Theorem}}
\newtheorem{remark}{{\bf Remark}}
\newtheorem{definition}[theorem]{{\bf Definition}}

\newtheorem{proposition}[theorem]{{\bf Proposition}}
\newtheorem{lemma}[theorem]{{\bf Lemma}}
\newtheorem{example}{{\bf Example}}

\def\bfv{\boldsymbol{v}}
\def\bfw{\boldsymbol{w}}

\def\res{{\rm res}}
\def\bsz{{\mathcal L}}
\def\OO{{\mathcal O}}
\def\sO{{\widetilde{\mathcal O}}}

\def\sOB{{\widetilde{\mathcal O}}_B}

\def\ZZ{\mathbb{Z}}

\begin{document}

\begin{frontmatter}
\title{Computing the topology of a plane or space hyperelliptic curve.}

\author[a]{Juan Gerardo Alc\'azar\fnref{proy,proy2}}
\ead{juange.alcazar@uah.es}
\author[a]{Jorge Caravantes\fnref{proy}}
\ead{jorge.caravantes@uah.e}
\author[b]{ Gema M. Diaz-Toca\fnref{proy}}
\ead{gemadiaz@um.es}
\author[c]{Elias Tsigaridas\fnref{proy4}}
\ead{elias.tsigaridas@inria.fr}

\fntext[proy]{Supported by the Spanish Ministerio de Ciencia, Innovaci\'on y Universidades and by the European Regional Development Fund (ERDF), under the project MTM2017-88796-P.}

\fntext[proy2]{Member of the Research Group {\sc asynacs} (Ref. {\sc ccee2011/r34}) }

\fntext[proy4]{Partially supported by a Giner de los R\'{\i}os Grant of the Universidad de Alcal\'a,  ANR JCJC GALOP (ANR-17-CE40-0009) and the PGMO grant ALMA.}

\address[a]{Departamento de F\'{\i}sica y Matem\'aticas, Universidad de Alcal\'a,
E-28871 Madrid, Spain}
\address[b]{Departamento de Ingenier\'ia y Tecnolog\'ia de Computadores, Universidad de Murcia, E-30100 Murcia, Spain} 
\address[c]{Inria Paris-Rocquencourt, Paris, France}
\begin{abstract}
We present algorithms to compute the topology of 2D and 3D hyperelliptic curves. The algorithms are based on the fact that 2D and 3D hyperelliptic curves can be seen as the image of a planar curve (the Weierstrass form of the curve), whose topology is easy to compute, under a birational mapping of the plane or the space. We report on a {\tt Maple} implementation of these algorithms, and present several examples. Complexity and certification issues are also discussed.  
\end{abstract}
\end{frontmatter}


\section{Introduction}\label{section-introduction}

Rational curves are widely used in Computer Aided Geometric Design. \emph{Hyperelliptic curves} are not rational, but they are \emph{birationally equivalent} to planar algebraic curves quadratic in one variable, the corresponding \emph{Weierstrass forms}, where birationally equivalent means that there exists a rational mapping between the curve and its Weierstrass form with an also rational inverse. Since Weierstrass forms are quadratic in one variable, hyperelliptic curves are parametrizable by square-roots. Thus, hyperelliptic curves are one of the simplest examples of non-rational families of curves. Furthermore, this type of curves appears frequently in Computer Aided Geometric Design. A good account of the occurrence of hyperelliptic curves in this field is given in \cite{Bizzarri}, where the problem of approximating hyperelliptic curves by means of rational parametrizations is addressed. As a brief summary of \cite{Bizzarri}, non-rational offsets of rational planar curves and some bisector curves (line$/$rational curve, or circle$/$rational curve) are planar hyperelliptic curves. Contour curves of canal surfaces, intersections of two quadrics or intersections of a quadric and a ruled surface are examples of hyperelliptic curves in 3-space. With more generality, every planar or space algebraic curve ${\mathcal C}$ admitting a square-root parametrization (see also \cite{SSV17}) is hyperelliptic. 

In this paper we address the problem of computing the topology of a
hyperelliptic curve ${\mathcal C}$. Efficient and fast algorithms to compute the Weierstrass form ${\mathcal G}$ of ${\mathcal C}$, as well as a birational mapping ${\bf x}:{\mathcal G}\dashrightarrow {\mathcal C}$ can be found in many computer algebra systems, e.g. Sage, Maple or Magma. Here we will assume that ${\bf x},{\mathcal G}$ are already known, and in fact that ${\mathcal C}$ is defined by means of the pair ${\bf x},{\mathcal G}$, so that ${\mathcal C}$ is seen as the image of the planar algebraic curve ${\mathcal G}$ under the mapping defined by ${\bf x}$. Since ${\mathcal G}$ is a simple curve, quadratic in one variable, and therefore the union of the graphs of two univariate functions, the topology of ${\mathcal G}$ is very easy to capture. Thus, our strategy to compute the topology of ${\mathcal C}$ is to study how the birational mapping modifies the topology of the Weierstrass form. Hence, we might say that the Weierstrass form ``guides" us to build the topology of ${\mathcal C}$. In more detail, we describe the topology of ${\mathcal G}$ by means of a \emph{topological graph} $G_{\mathcal G}$, i.e. a graph isotopic to the curve. Then the topology of ${\mathcal C}$ is described by means of another graph $G_{\mathcal C}$ whose vertices are the images of the vertices of 
$G_{\mathcal G}$ under ${\bf x}$, and whose edges correspond to the branches of ${\bf x}({\mathcal G})$, which are in one-to-one correspondence with the edges of $G_{\mathcal G}$. If ${\bf x}$ becomes infinite at a vertex of $G_{\mathcal G}$, the image of such a vertex corresponds to a branch at infinity of ${\mathcal C}$.

Additionally, the pair ${\bf x},{\mathcal G}$ may come for free, or almost for free, in certain applications; see for instance the introductory example of an intersection curve at the beginning of Section 2. If the pair ${\bf x},{\mathcal G}$ is known, in order to determine the topology
of ${\mathcal C}$ one might compute an implicit representation of
${\mathcal C}$ using elimination methods. This yields one implicit
equation in the plane case, and at least two implicit equations in the space case. In both cases, plane and space, after computing the implicit equation(s)
one might use existing algorithms to find the topology of the curve:
see for instance \cite{Berb, CLPPRT09, EKW07, GVN}, among many others,
for the planar case, or \cite{AS05, Lazard13, Daouda, Kahoui} for the
space case. However, such an implicit representation typically has a
high degree and big coefficients, which makes it difficult to use.
Moreover, many algorithms have additional
  assumptions, for example generic position, or complete intersection in the space case,
  that are computationally expensive to fulfill.
As a consequence, if the pair ${\bf x},{\mathcal G}$ is known, it
is useful to have an alternative method for computing the topology of
${\mathcal C}$ that avoids using an implicit representation. 

On the other hand, if ${\mathcal C}$ is defined by means of an implicit representation the pair ${\bf x},{\mathcal G}$ can be computed using a computer algebra system. Thus, our algorithm is applicable to that case as well, and provides an alternative to existing algorithms for computing the topology of a plane or space curve. This is specially useful in the space case, since known algorithms to compute the topology of a space case are not so easy to use in practice, and have a high complexity (see Section 6.3).

It is worth comparing our paper with some other related papers. In \cite{ADT10} the topology of 2D and 3D rational curves is addressed. In \cite{ADT10} the curve is seen as the image of the real line under a planar or space birational mapping, so somehow the germ of the idea in this paper is already in \cite{ADT10}. In \cite{Caravantes}, a method to compute the topology of a (non-necessarily rational) offset curve of a rational planar curve is provided. The method exploits similar ideas to \cite{ADT10}, but focuses on offset curves, which have special properties. Finally, in \cite{Bizzarri} the problem of approximating a hyperelliptic curve by means of rational curves is considered. The Weierstrass form is also used in \cite{Bizzarri}, but the goal is different, and in particular the computation of the topology of the hyperelliptic curve is not addressed. 

Our method has been implemented in the computer algebra system {\tt Maple} 2017, and the implementation can be freely downloaded from \cite{gmdt}. In order to certify the topology we need to certify self-intersections, i.e., we need to certify whether or not the image of two points under the birational mapping giving rise to our curve, is the same. This requires to work with algebraic numbers, and is computationally difficult. We address this problem, and we provide a complexity analysis of the algorithm with and without the certification step. While the complexity bound that we get is not better than the known complexity for the implicit planar case \cite{Sagra15}, it is, however, definitely better compared to the implicit space case \cite{Diatta, Lazard13}. It is true, however, that in \cite{Diatta, Lazard13} the space curve is assumed to be given by an implicit representation. However, in our paper, even though the algorithm is applicable also to implicit curves after computing a Weierstrass form of the curve (which is efficient and fast), we assume a different representation of the curve, namely as the birational image of a Weierstrass curve.

The structure of this paper is the following. We motivate and present the problem in Section \ref{sec-prelim}, where some preliminary notions and ideas are given. The planar case is addressed in Section 3, and the space case is studied in Section 4. In Section 5 we report on the results of our experimentation, carried out in the computer algebra system {\tt Maple 2017}. In Section 6, we address the complexity of the algorithm, we consider certification issues, and we compare the complexity of our algorithm with the known complexities of algorithms using an implicit representation of the curve. Section 7 contains our conclusions. The proofs of some results in Section 3 are postponed to Appendix I, so as not to stop the flow of the paper.

\section{Motivation and presentation of the problem.} \label{sec-prelim}

Consider a \emph{biquadratic} patch $S$, commonly used in Computer Aided Geometric Design, parametrized by
\begin{equation}\label{patch}
{\bf x}(t,s)=(x(t,s),y(t,s),z(t,s))=\sum_{i=0}^2 \sum_{j=0}^2 {\bf c}_{ij} B_i(t)B_j(s),
\end{equation}
where $B_k(u)={2\choose k} u^k(1-u)^{2-k}$ for $k=0,1,2$, and ${\bf c}_{ij}\in {\Bbb R}^3$ for $i,j=0,1,2$. Assume that we want to describe the topology of the intersection curve ${\mathcal C}$ of $S$ with a general plane $\Pi$ of equation $Ax+By+Cz+D=0$, i.e. the topology of $S\cap \Pi$. In order to do this, substituting the components $x(t,s)$, $y(t,s)$, $z(t,s)$ of ${\bf x}(t,s)$ into the equation of $\Pi$ we get an algebraic condition $g(t,s)=0$; since the components of ${\bf x}(t,s)$ have bidegree $(2,2)$, one can see that 
\begin{equation}\label{ex-hyper}
g(t,s)=\Psi_1(t)s^2+\Psi_2(t)s+\Psi_3(t)=0,
\end{equation}
where the $\Psi_i(t)$, $i=1,2,3$, are polynomials in the variable $t$. Then the curve ${\mathcal C}=S\cap \Pi$ can be described as the closure of the image of the planar curve ${\mathcal G}$, defined by $g(t,s)=0$  in the $(t,s)$-plane, under the (rational) mapping ${\bf x}$, i.e. ${\mathcal C}=\overline{{\bf x}({\mathcal G})}$. Notice that $\mathcal{C}-{\bf x}(\mathcal{G})$ reduces to finitely many points corresponding to either the image of points of ${\mathcal G}$ at infinity, or limit points in ${\mathcal C}$ corresponding to base points of ${\bf x}$, lying in ${\mathcal G}$.

\vspace{0.3 cm}
The situation presented above is an example of the general problem treated in this paper. Given a planar curve ${\mathcal G}$, implicitly defined in the plane $(t,s)$ by a polynomial equation like Eq. \eqref{ex-hyper}, of degree 2 in the variable $s$, our goal is to compute the topology of the curve ${\mathcal C}=\overline{{\bf x}({\mathcal G})}$, where ${\bf x}:{\Bbb R}^2\to {\Bbb R}^n$, with $n=2$ or $n=3$, is \emph{birational} when restricted to ${\mathcal G}$; in particular, in that case the inverse mapping ${\bf x}|_{\mathcal G}^{-1}: {\mathcal C}\to {\mathcal G}$ exists and is rational. Writing
\[
{\bf x}=(x_1,x_2,\ldots,x_n),
\]
we will refer to the functions $x_i:{\Bbb R}^2\to {\Bbb R}$ as the \emph{components} of the mapping ${\bf x}$. Notice that if ${\mathcal C}$ is a rational curve, in which case the curve ${\mathcal G}$ must also be rational because of the birationality of the mapping ${\bf x}|_{\mathcal G}$, then the problem can be solved using already existing methods \cite{ADT10}. Thus, we will assume that ${\mathcal C}$, and therefore also ${\mathcal G}$, is not rational, in which case ${\mathcal C}$ is said to be a \emph{hyperelliptic curve}. 

\vspace{0.3 cm}
With some generality (see for instance \cite{Bizzarri}), we say that a curve ${\mathcal C}$ is \emph{hyperelliptic} if there exists a generically two-to-one map ${\mathcal C}\to {\Bbb R}$. Furthermore, such a curve (see for instance \cite{hyp}) is birationally equivalent to a planar curve 
\begin{equation}\label{weierstrass}
s^2-p(t)=0,
\end{equation} 
where $p(t)$ is a square-free polynomial of degree $2{\bf g}+1$ or $2{\bf g}+2$, where ${\bf g}$ is the \emph{genus} of ${\mathcal C}$. Recall (see for instance \cite{SWP}) that the genus ${\bf g}$ is a birational invariant that, in particular, characterizes rational curves: ${\bf g}=0$ corresponds to rational curves, while for non-rational curves ${\bf g}\geq 1$, ${\bf g}\in {\Bbb N}$. Additionally, whenever we work over a field of characteristic different from 2, as it is our case, one can always get a Weierstrass curve where the degree of $p(t)$ is $2{\bf g}+1$ (see for instance \cite{hyp}). Also, Eq. \eqref{weierstrass} is called the \emph{Weierstrass form} of ${\mathcal C}$. Notice (see p. 59 of \cite{Bizzarri}) that we can always transform the expression Eq. \eqref{ex-hyper} of our motivating example into an expression like Eq. \eqref{weierstrass} by considering a change of parameters
\[
t:=t,\mbox{ }s:=\frac{-B(t)+s}{2A(t)}.
\]
Furthermore, the reached expression corresponds to a hyperelliptic curve if and only if the polynomial $p(t)$ in the expression we obtain is square-free (see Lemma 2 of \cite{hyp}).

In this paper we will assume that the Weierstrass form has already been computed, and therefore that the curve ${\mathcal G}$ is described by means of Eq. \eqref{weierstrass}. Additionally, we will assume that the curve ${\mathcal G}$ is real, i.e. that it contains infinitely many real points; if ${\mathcal G}$ is not real, then because of the birationality of ${\bf x}|_{\mathcal G}$, ${\mathcal C}$ cannot be real either. Observe also that since $s^2-p(t)$ is an irreducible polynomial in $t,s$, so is the curve ${\mathcal G}$; since irreducibility is a birational invariant, we deduce that ${\mathcal C}$ is irreducible as well.

\vspace{0.3 cm}
In order to describe the topology of the curve ${\mathcal C}$, we will compute, as it is common, a graph \emph{ambient isotopic} to ${\mathcal C}$. 

\begin{definition}\label{isotop}
An \emph{ambient isotopy} between two subspaces $X,Y$ of ${\Bbb R}^n$ is a continuous function $H:{\Bbb R}^n\times [0,1]\to {\Bbb R}^n$ satisfying the following conditions: (1) $H(\bullet;0)$ is the identity; (2) $H(X;1)=Y$; (3) for all $\omega\in [0,1]$, $H(\bullet; \omega)$ is a homeomorphism from ${\Bbb R}^n$ to ${\Bbb R}^n$. We say that $X,Y\subset {\Bbb R}^n$ are \emph{isotopic}, if there exists an (ambient) isotopy $H$ satisfying these conditions.
\end{definition}

If $X,Y$ in Definition \ref{isotop} are 1-dimensional objects, the fact that $X,Y$ are ambient isotopic implies that one of them can be deformed into the other without removing or introducing self-intersections (see for instance \cite{Hirsch}). For simplicity, in the rest of the paper we will omit the word ``ambient" whenever we speak about isotopy.  Now we have the following definition.

\begin{definition}\label{topgraph}
Let ${\mathcal C}\subset {\Bbb R}^n$, where $n=2$ or $n=3$. A \emph{topological graph} of ${\mathcal C}$ is a graph $G_{\mathcal C}$ isotopic to ${\mathcal C}$ whose vertices lie on the curve ${\mathcal C}$.
\end{definition}

\begin{remark} Vertices of $G_{\mathcal C}$ with valence equal to one, i.e. belonging only to one edge, correspond to real branches of ${\mathcal C}$ at infinity. Thus, if $G_{\mathcal C}$ contains some vertex of this type, then ${\mathcal C}$ is not bounded. 
\end{remark}

Thus, our goal is to build an algorithm for computing a topological graph $G_{\mathcal C}$ of ${\mathcal C}$; we will refer to $G_{\mathcal C}$ as the graph \emph{associated with} ${\mathcal C}$. In order to do this, we will not compute $G_{\mathcal C}$ directly: instead, we will compute a graph $G_{\mathcal G}$ associated with ${\mathcal G}$, and we will derive $G_{\mathcal C}$ from $G_{\mathcal G}$ by studying how the topology of ${\mathcal G}$ changes when ${\bf x}$ is applied. Furthermore, in our analysis we do not consider isolated real points of ${\mathcal C}$, which can be generated by complex branches of ${\mathcal G}$ at infinity.
Let us briefly recall how graphs associated with planar and space curves are computed.

\vspace{0.3 cm}
\noindent \underline{\sc Graph associated with a planar curve.}

\vspace{3 mm}
Let $f(x,y)=0$ define a planar algebraic curve ${\mathcal F}$ without vertical asymptotes. We say that $P\in {\mathcal F}$ is \emph{regular} if either $f_x(P)\neq 0$ or $f_y(P)\neq 0$; otherwise, we say that $P$ is \emph{singular}. We say that $P\in {\mathcal F}$ is \emph{critical} if $P$ satisfies that $f(P)=f_y(P)=0$. A critical point which is not singular is called a \emph{ramification} point. The topological graph $G_f$ associated with ${\mathcal F}$ can be described as follows (see Fig. \ref{fig-graphs}, left): 

\begin{itemize}
\item The {\bf vertices} of the graph $G_f$ are: (1) the critical points of ${\mathcal F}$; (2) the points of ${\mathcal F}$ lying on the vertical lines through the critical points of ${\mathcal F}$ (we call these vertical lines, \emph{critical lines}); (3) the points of ${\mathcal F}$ lying on vertical lines placed: (3.1) between two consecutive critical lines, (3.2) at the left of the left-most critical point, and (3.3) at the right of the right-most critical point.  
\item Two vertices of $G_f$ are connected by an {\bf edge} of $G_f$ iff there is a real branch of ${\mathcal F}$ connecting the corresponding points on ${\mathcal F}$. 
\end{itemize}

The problem of computing a topological graph of an implicit planar curve is well-studied. The interested reader can check the references \cite{Berb, CLPPRT09, EKW07, GVN}, among others, for further information on the problem. Although it is customary, in most papers dealing with the problem of computing the graph $G_f$, to start with the assumption that ${\mathcal F}$ does not have vertical asymptotes or vertical components, one can adapt the strategy without assuming these properties; see for instance \cite{EKBS}.

\vspace{0.3 cm}
\noindent \underline{\sc Graph associated with a space curve.}

\vspace{3 mm}
Let $\{f_1(x,y,z)=0,\ldots,f_m(x,y,z)=0\}$ define a space algebraic curve ${\mathcal F}$: (i) without asymptotes parallel to the $z$-axis; (ii) without components parallel to the $z$-axis; (iii) such that the projection $\pi_{xy}({\mathcal F})$ of ${\mathcal F}$ onto the $xy$-plane is birational. Hypothesis (iii) ensures that there are not two different real branches of ${\mathcal F}$ projecting onto a same branch of $\pi_{xy}({\mathcal F})$. Taking advantage of Hypothesis (iii), the usual strategy to compute a topological graph $G_f$ isotopic to ${\mathcal F}$ is to birationally project ${\mathcal F}$ onto some plane, say, the $xy$-plane, then compute a graph isotopic to the projection $\pi_{xy}({\mathcal F})$, which is a planar algebraic curve, and later ``lift" the graph associated with $\pi_{xy}({\mathcal C})$ to a space graph. Since the projection $\pi_{xy}$ is birational, one can be sure that every edge of the graph associated with $\pi_{xy}({\mathcal F})$ lifts to one, and just one, edge of the graph associated with ${\mathcal F}$. More precisely, the graph $G_f$ associated with ${\mathcal F}$ can be described as follows (see Fig. \ref{fig-graphs}, right):  

\begin{itemize}
\item The {\bf vertices} of the graph $G_f$ are the points of ${\mathcal F}$ projecting as vertices of the graph associated with $\pi_{xy}({\mathcal F})$.
\item Two vertices of $G_f$ are connected by an {\bf edge} of $G_f$ iff the corresponding points of ${\mathcal C}$ are connected by a real branch of ${\mathcal C}$. Furthermore, if the vertices are not singularities of $\pi_{xy}({\mathcal C})$, we connect them iff their projections are connected in the graph associated with $\pi_{xy}({\mathcal F})$. For vertices corresponding to singularities of $\pi_{xy}({\mathcal C})$ the process is more complicated, since we can have two non-overlapping branches of ${\mathcal C}$ whose projections onto the $xy$-plane overlap (see Fig. \ref{fig-graphs}, left); for references on how to deal with this problem, one can check \cite{Daouda, Kahoui}. 
\end{itemize}

\begin{figure}
$$\begin{array}{c}
  \includegraphics[scale=0.4]{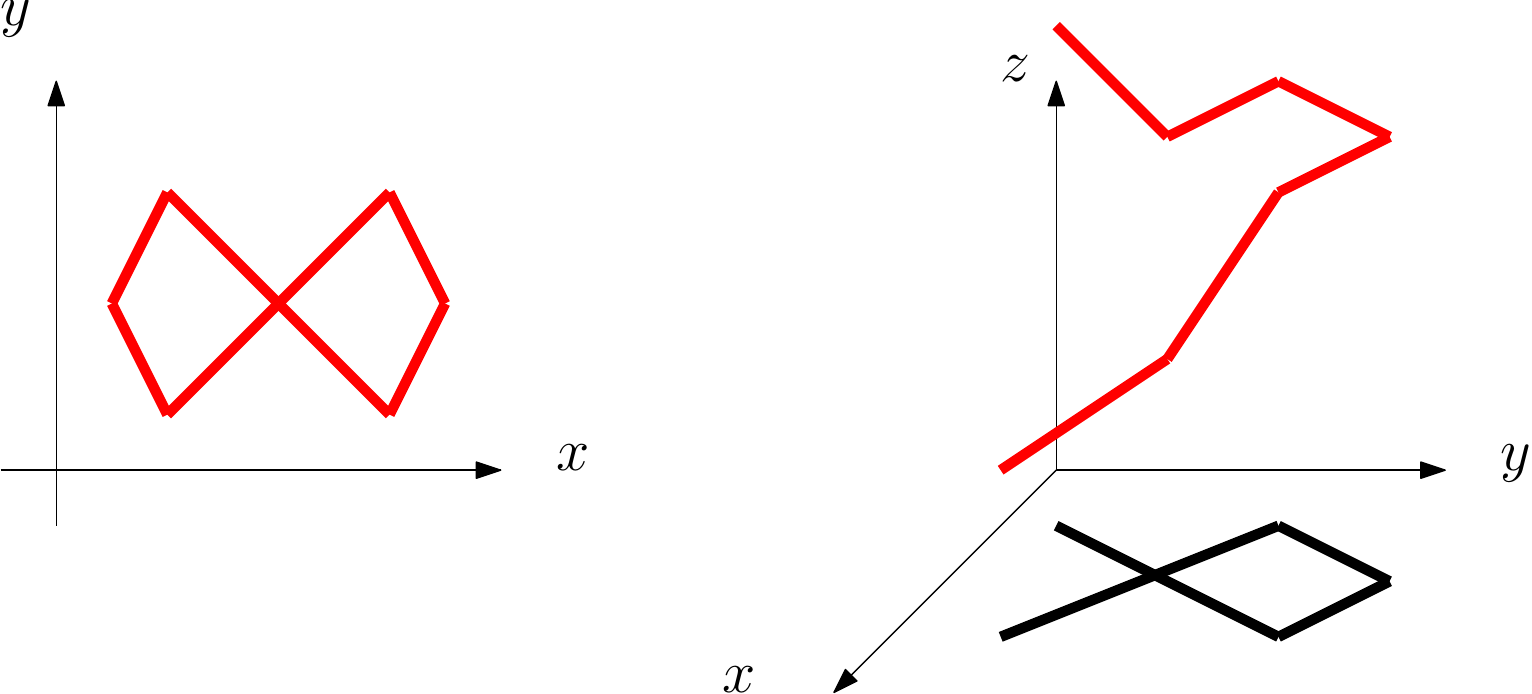}
 \end{array}$$
 \caption{Graphs associated with planar and space curves}\label{fig-graphs}
\end{figure}

The problem of computing a topological graph associated with an implicit space algebraic curve has received some attention in the literature, although less than the planar case. The interested reader can check the references \cite{AS05, Lazard13, Daouda, Kahoui} for more details on the problem. Again, as it also happens in the planar case, the strategy can be adapted to the case when ${\mathcal F}$ has vertical components or vertical asymptotes. 

\vspace{0.3 cm}
\noindent \underline{\sc In our case.}

\vspace{3 mm}

In our case, we need to compute the graph $G_{\mathcal G}$ associated with ${\mathcal G}$ plus some extra vertices $Q_i=(t_i,s_i)\in {\mathcal G}$. In particular, we need to include points $Q_i\in {\mathcal G}$ giving rise to certain notable points $P_i\in {\mathcal C}$, as we will see in the next sections. And we also need to include the points $Q_i\in {\mathcal G}$ where some component of ${\bf x}$ has the indeterminacy $\frac{0}{0}$, or becomes infinite. After including these vertices, we observe that ${\bf x}$ is continuous over each portion of the curve ${\mathcal G}$ corresponding to each edge of $G_{\mathcal G}$. Then, the key idea is that since the image of any connected subset of ${\mathcal G}$ is also connected, every edge $e$ of $G_{\mathcal G}$ gives rise to an edge $\tilde{e}$ of $G_{\mathcal C}$, namely the edge connecting the images of the vertices of $e$. Hence, the topology of ${\mathcal G}$ guides us to compute the topology of ${\mathcal C}$. 

The fact that ${\bf x}$ is birational over ${\mathcal G}$ guarantees that all the edges of $G_{\mathcal C}$ are obtained this way, since there cannot be any real branch of ${\mathcal C}$ coming from a complex branch of ${\mathcal G}$: indeed, if ${\mathcal B}\subset {\mathcal G}$ is a complex branch such that ${\bf x}({\mathcal B})$ is real, then ${\bf x}({\mathcal B})=\overline{{\bf x}({\mathcal B})}$, where $\overline{{\bf x}({\mathcal B})}$ denotes the conjugate of ${\bf x}({\mathcal B})$. But then there are infinitely many points of ${\mathcal C}$ with at least two pre-images, which cannot happen because ${\bf x}|_{\mathcal G}$ is birational. 

Therefore, the rough idea in order to build $G_{\mathcal C}$ is to compute the graph $G_{\mathcal G}$ (by using any of the well-known algorithms to do this), and the images $P_i$ of the vertices $V_i$ of $G_{\mathcal G}$. Then we connect the $P_i$ according to how their preimages $V_i={\bf x}|_{\mathcal G}^{-1}(P_i)$ are connected in ${\mathcal G}$. If some component of ${\bf x}(V_i)$ becomes infinite, then we have an open branch of ${\mathcal C}$, i.e. a branch of ${\mathcal C}$ going to infinity; in particular, in that case ${\mathcal C}$ is not bounded.

Fig. \ref{fig-idea} represents the idea of computing $G_{\mathcal C}$ from $G_{\mathcal G}$, for the case $n=2$: each edge, marked with a different color, of the graph $G_{\mathcal G}$ (left), gives rise to an edge, marked with the same color, of the graph $G_{\mathcal C}$ (right). 

Observe that since ${\mathcal G}$ is implicitly defined by Eq. \eqref{weierstrass}, the leading coefficient in the variable $s$ is constant, so ${\mathcal G}$ has no asymptotes parallel to the $s$-axis, which we take as the vertical axis in the $(t,s)$ plane. Additionally, since the Weierstrass form implies that $p(t)$ is square-free, one can see that ${\mathcal G}$ is regular, and that the only critical points are the points $\{s=0,p(t)=0\}$, all of which are ramification points, i.e. points where the tangent line to ${\mathcal G}$ is vertical. Because of this, ${\mathcal G}$ consists of open branches and$/$or closed components, without self-intersections. As a projective variey, though, ${\mathcal G}$ has a singular point, namely the point at infinity of ${\mathcal G}$ (in the direction of the $s$-axis).

Certainly, there can also be some points of ${\mathcal C}$ which do not belong to ${\bf x}({\mathcal G})$. The points in ${\mathcal C}-{\bf x}({\mathcal G})$ correspond to the images of the point at infinity of ${\mathcal G}$, and the limit points coming from the base points of ${\bf x}$ lying in ${\mathcal G}$, i.e. points of ${\mathcal G}$ where all the numerators and denominators of the components of ${\bf x}$ vanish simultaneously. Since ${\mathcal G}$ is regular over its affine part, we can be sure that ${\bf x}$ extends to its base points (see Theorem 1.2 of \cite{Shafarevich}), so that base points give rise to either affine points of ${\mathcal C}$, or points at infinity of ${\mathcal C}$. The effective computation of the images of base points of ${\bf x}$ on ${\mathcal G}$ is analyzed in the next section. On the other hand, ${\mathcal G}$ 
has one singular point at infinity with two different branches, i.e. two different \emph{places} centered at this point (see \cite{walker} for further information on places). This implies that the point at infinity of ${\mathcal G}$ can give rise to at most two points of ${\mathcal C}$, that can be affine, or at infinity. We denote these points by $P_{\infty},P_{-\infty}$, that may or may not coincide. This notation responds to the fact that these points are reached by analyzing the behavior of ${\bf x}|_{\mathcal G}$ when $t\to \infty$ and $t\to -\infty$. In the next section, we will consider the computation of these points, that we will represent in a more compact way by $P_{\pm \infty}$. 

\begin{figure}
$$\begin{array}{c}
  \includegraphics[scale=0.5]{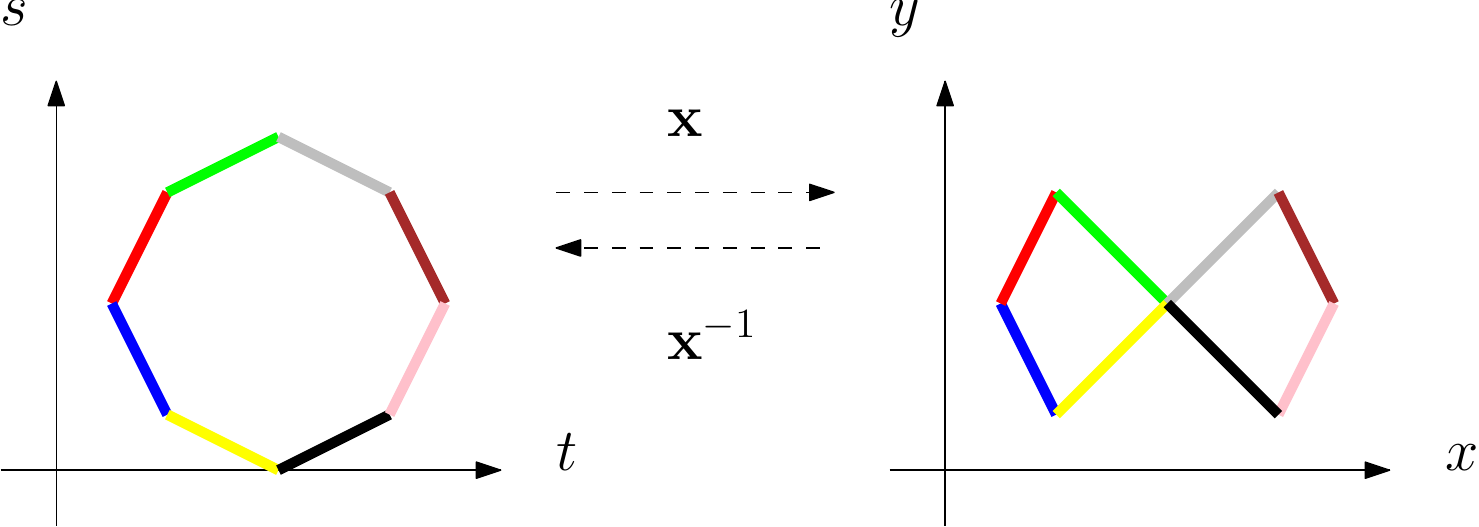}
 \end{array}$$
 \caption{$G_{\mathcal G}$ and $G_{\mathcal C}$}\label{fig-idea}
\end{figure}

\section{The planar case. }\label{sec-planar}

Let ${\bf x}:{\Bbb R}^2\to {\Bbb R}^2$, where 
\[
{\bf x}(t,s)=(x(t,s),y(t,s))=\left (\frac{A_1(t,s)}{B_1(t,s)},\frac{A_2(t,s)}{B_2(t,s)}\right),
\]
and let ${\mathcal C}={\bf x}({\mathcal G})$, where ${\mathcal G}$ is implicitly defined by an equation $g(t,s)=s^2-p(t)=0$ like Eq. \eqref{weierstrass}. The functions $x(t,s),y(t,s)$ are the \emph{components} of ${\bf x}(t,s)$. We require ${\bf x}$ to be a rational mapping satisfying that the restriction ${\bf x}|_{\mathcal G}$ is birational, so that ${\bf x}|_{\mathcal G}^{-1}:{\mathcal C}\to {\mathcal G}$ is well-defined, and therefore rational. We can always check this assumption with a probabilistic algorithm; we take a random point $(t_0,s_0)\in {\mathcal G}$, compute the point $P={\bf x}(t_0,s_0)$, and finally determine the preimages of ${\bf x}(t_0,s_0)$: if we get only one preimage belonging to ${\mathcal G}$, then with probability one the required hypothesis holds. Additionally, using repeatedly the fact that $s^2=p(t)$, we can write 
${\bf x}|_{\mathcal G}(t,s)$ in the following form: 
\begin{equation}\label{ofx} 
{\bf x}|_{\mathcal G}(t,s)=\left (\frac{A_1(t,s)}{B_1(t,s)},\frac{A_2(t,s)}{B_2(t,s)}\right)=\left(\frac{a_{11}(t)+sa_{12}(t)}{b_{11}(t)+sb_{12}(t)},\frac{a_{21}(t)+sa_{22}(t)}{b_{21}(t)+sb_{22}(t)}\right),
\end{equation}
where we can assume that $A_i,B_i$ are relatively prime for $i=1,2$. Observe that this implies $\gcd(a_{11},a_{12},b_{11},b_{12})=1$ and $\gcd(a_{21},a_{22},b_{21},b_{22})=1$. Notice also that in general $b_{11}(t)\neq b_{21}(t)$, $b_{12}(t)\neq b_{22}(t)$.

\vspace{0.3 cm}
As observed in Section \ref{sec-prelim}, we first need to describe the topology of ${\mathcal G}$ by means of a graph $G_{\mathcal G}$ isotopic to it, with some additional vertices. We need to include the following points as vertices of $G_{\mathcal G}$:

\begin{itemize}
\item [(i)] \emph{Critical points of $g(t,s)=0$}, i.e. points of ${\mathcal G}$ where $g_s=0$. 
\item [(ii)] \emph{Points of ${\mathcal G}$ giving rise to critical points of ${\mathcal C}$}.
\item [(iii)] \emph{Points of $\mathcal{G}$ where some component of ${\bf x}$ is not defined}.
\item [(iv)] \emph{Starting and ending points for open branches of ${\mathcal G}$}.
\end{itemize}

The points in (i) are the solutions of $g=g_s=0$, i.e. the points $\{s=0,p(t)=0\}$. The points in (iv) can be easily computed by taking a $t$-value at the left (resp. right) of the left-most (resp. the right-most) solution of $g=g_s$. The points in (iii) are the points $(t,s)\in {\mathcal G}$ such that $B_1(t,s)\cdot B_2(t,s)=0$. In particular, some of the points in (iii) may generate asymptotes of ${\mathcal C}$; also, \emph{base points} of ${\bf x}$ in ${\mathcal G}$, i.e. the points of ${\mathcal G}$ where 
\[
A_1(t,s)=B_1(t,s)=A_2(t,s)=B_2(t,s)=g(t,s)=0,
\]
are included in (iii). The topology of ${\mathcal G}$ is easy to capture (see for instance \cite{Bizzarri}), and can be computed by using known algorithms for planar curves like \cite{Berb, CLPPRT09, EKW07, GVN}.

\subsection{Computing the points of ${\mathcal G}$ giving rise to critical points of ${\mathcal C}$}

For simplicity, in this section we will assume that ${\bf x}$ has no base points on ${\mathcal G}$. These points, which may also generate critical points of ${\mathcal C}$, will be analyzed in the next subsection. Some observations on how to use the results in this subsection in the presence of base points will be done at the end of the subsection. Additionally, if the points $P_{\pm \infty}$ are affine they may be critical points of ${\mathcal C}$ as well. The behavior of $P_{\pm \infty}$ will be studied in Subsection \ref{subsec-nueva}.

Now in Section \ref{sec-prelim} we recalled that the critical points of ${\mathcal C}$ are either singularities, or ramification points, i.e. points where the tangent line is vertical. It is useful to distinguish two types of singularities : \emph{local singularities}, which correspond to singular points $P\in {\mathcal C}$ with just one branch of ${\mathcal C}$ through $P$, and \emph{self-intersections} of ${\mathcal C}$, which correspond to points $P\in {\mathcal C}$ with at leat two different branches of ${\mathcal C}$ through $P$. In Fig. \ref{localsingu} we show three examples of local singularities, two of them cuspidal (first two curves, starting from the left) and one of them non-cuspidal (third curve, starting from the left), and a self-intersection (right-most curve); see \cite{AS07} for more information on local singularities.

\begin{figure}[h]
$$
\begin{array}{cccc}
\includegraphics[scale=0.17]{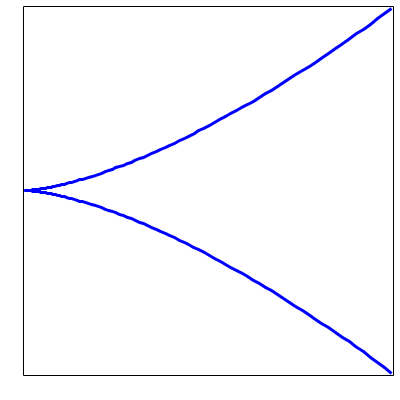}&
\includegraphics[scale=0.17]{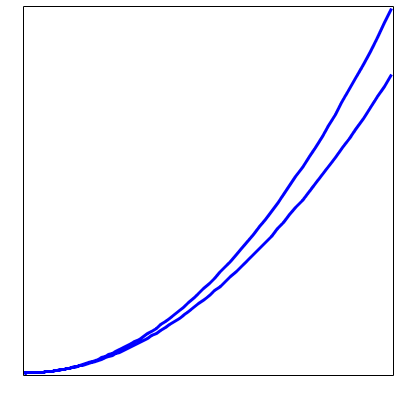}&
\includegraphics[scale=0.17]{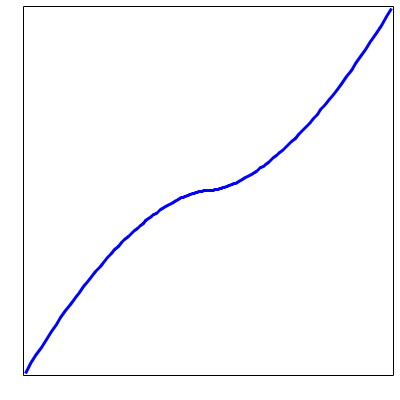}&
\includegraphics[scale=0.17]{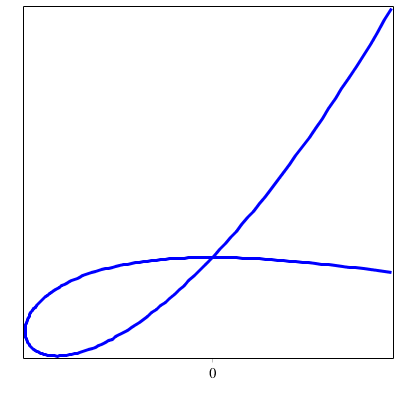}
\end{array}
$$
\caption{Local singularities (three local singularities and q self-intersection (right-most curve).}\label{localsingu}
\end{figure}

In order to compute the points of ${\mathcal G}$ giving rise to local singularities and ramification points of ${\mathcal C}$, we analyze ${\bf x}({\mathcal G})$, where ${\mathcal G}$ is implicitly defined by $g(t,s)=0$. The differential of ${\bf x}$ defines a mapping between the tangent space to ${\mathcal G}$ and the tangent space to ${\mathcal C}$, at corresponding points. Denoting a generic element of the tangent space to ${\mathcal C}$ by $\bfv=(v_1,v_2)$, we have the following relationship; here, $x_t$ represents the partial derivative of $x(t,s)$ with respect to the variable $t$, and similarly for $y_t,x_s,y_s,g_t,g_s$:

\begin{equation} \label{tangent-space}
\begin{bmatrix}
x_t & x_s\\
y_t & y_s
\end{bmatrix}
\cdot 
\begin{bmatrix}
g_s \\
-g_t
\end{bmatrix}=
\begin{bmatrix}
v_1\\
v_2
\end{bmatrix}
\end{equation}

The above relationship follows from differentiating with respect to $t$ the components of ${\bf x}|_{\mathcal G}$. Whenever $g_s\neq 0$ (i.e. whenever $(t,s)$ is not a ramification point of ${\mathcal G}$), $g(t,s)=0$ implicitly defines a differentiable function $s=s(t)$, where $\frac{ds}{dt}=-\frac{g_t}{g_s}$. Now differentiating ${\bf x}(t,s)=0$ where $s=s(t)$ is the function defined by $g(t,s)=0$, and using the Chain Rule, we get a vector $\bfw$ which is parallel to the vector $\bfv$ in Eq. \eqref{tangent-space}. For the points where $g_s=0$, we can proceed in the same way, reaching the same result, differentiating with respect to $s$ instead. Since all affine points of ${\mathcal G}$ are regular, i.e. either $g_t$ or $g_s$ are nonzero, Eq. \eqref{tangent-space} holds.

\begin{lemma} \label{local}
Suppose that ${\bf x}$ has no base points lying on ${\mathcal G}$, and let $P\in {\mathcal C}$, $P\neq P_{\pm \infty}$, $P={\bf x}(t_0,s_0)$, where $(t_0,s_0)\in {\mathcal G}$. If $P$ is a either a local singularity or a ramification point of ${\mathcal C}$, then $(t_0,s_0)$ satisfies that
\begin{equation}\label{eq-local}
g=x_tg_s-x_sg_t=0.
\end{equation}
\end{lemma}

\begin{remark} \label{rem-sing} 
For the local singularities we have
\begin{equation}\label{eq-local}
g=x_tg_s-x_sg_t=y_tg_s-y_sg_t=0.
\end{equation}
\end{remark}

However, Lemma \ref{local} does not necessarily provide the self-intersections of ${\mathcal C}$. In order to find these last singularities, we imitate the strategy in \cite{ADC15}. First we define 
\begin{equation} \label{eq-xis}
\begin{array}{rcl}
\xi_1(x,t) &=& \mbox{square-free part of }\mbox{Res}_s(\mbox{num}(x-x(t,s)),g(t,s)),\\
\xi_2(x,y,t) &=& \mbox{square-free part of }\mbox{Res}_s(\mbox{num}(x-x(t,s)),\mbox{num}(y-y(t,s))),
\end{array}
\end{equation}

\noindent where $\mbox{num}(\bullet)$ denotes the numerator of the rational function $\bullet$. Notice that in general, eliminating $t$ in $\xi_1(x,t)=0,\mbox{ }\xi_2(x,y,t)=0$ by means of the resultant $\mbox{Res}_t(\xi_1(x,t),\xi_2(x,y,t))$, we obtain a polynomial in $x,y$ containing, as a factor, the implicit equation of ${\mathcal C}$. Using the definition of the resultant, one can easily check that $\xi_1(x,t)$ is a quadratic polynomial in $x$, and $\xi_2(x,y,t)$ is quadratic as a polynomial in $x,y$, and linear in $x$ and in $y$ (i.e. $\xi_2(x,y,t)$ is bilinear). 

Now the key idea to find the self-intersections of ${\mathcal C}$ is that these points are among the points $(x,y)\in {\mathcal C}$ where $t={\bf x}|_{\mathcal G}^{-1}(x,y)$ is not defined. For a generic point $(x_0,y_0)\in {\mathcal C}$, we can find $t_0={\bf x}|_{\mathcal G}^{-1}(x_0,y_0)$ as the \emph{only} root of $\mbox{gcd}(\xi_1(x_0,t),\xi_2(x_0,y_0,t))$. In order to find the \emph{function} $t=t(x,y)={\bf x}|_{\mathcal G}^{-1}(x,y)$, we can compute the gcd of $\xi_1(x,t)$ and $\xi_2(x,y,t)$ as polynomials in the variable $t$ whose coefficients are real polynomials in $x,y$, with the additional condition $f(x,y)=0$, where $f$ is the implicit equation of ${\mathcal C}$. More formally, one sees $\xi_1(x,t)$ and $\xi_2(x,y,t)$ as elements of ${\Bbb R}({\mathcal C})[t]$, where ${\Bbb R}({\mathcal C})$ is the field of real rational functions of ${\mathcal C}$. Since ${\mathcal C}$ is irreducible ${\Bbb R}({\mathcal C})$ is a Euclidean domain. Therefore 
\[D(x,y,t)=\gcd_{{\Bbb R}({\mathcal C})[t]}(\xi_1,\xi_2)\]
is well-defined and can be computed, for instance, by means of the Euclidean algorithm. Since ${\bf x}|_{\mathcal G}$ is proper, $D(x,y,t)$ is linear in $t$ and solving $D(x,y,t)=0$ for $t$, one gets $t={\bf x}|_{\mathcal G}^{-1}(x,y)$. 

Following the ideas of \cite{ADC15}, one can compute ${\bf x}|_{\mathcal G}^{-1}(x,y)$ more eficiently as follows (see \cite{ADC15} for further detail). By the fundamental property of subresultants, $D(x,y,t)$ is the first subresultant different from zero (modulo $f(x,y)$) in the subresultant chain of $\xi_1,\xi_2$, seen as elements of the domain ${\Bbb R}[x,y][t]$. If the degrees of $\xi_1,\xi_2$ as elements of ${\Bbb R}[x,y][t]$ are $n_1,n_2$, the elements of the subresultant chain are represented as 
\[
\{\mathbf{Subres}_i(\xi_1,n_1,\xi_2,n_2)_{i\geq 0}\},
\]
with $0\leq i \leq \mbox{inf}(n_1,n_2)-1$, and can be defined as determinants of order $n_1+n_2-i$ of Sylvester-like matrices whose entries are related to the coefficients of $\xi_1,\xi_2$ (see Section 2.2 of \cite{ADC15}). Since $\mbox{deg}(\mathbf{Subres}_i(\xi_1,n_1,\xi_2,n_2))\leq i$, and by the birationality of ${\bf x}|_{\mathcal G}$ we have $\mbox{deg}(G(x_0,y_0,t))=1$ for almost all $(x_0,y_0)\in {\mathcal C}$, we deduce that $D(x,y,t)$ is equal to $\mathbf{Subres}_1(\xi_1,n_1,\xi_2,n_2)$; notice that $\mathbf{Subres}_1(\xi_1,n_1,\xi_2,n_2)$ can be computed without actually knowing the implicit equation of ${\mathcal C}$. Writing
\[
\begin{array}{c}
\mathbf{Subres}_1(\xi_1,n_1,\xi_2,n_2)(t)=\mathbf{sres}_1(x,y)\,t+\mbox{sr}_1(x,y),  
\end{array}
\]
we have that 
\begin{equation} \label{t-inv}
t={\bf x}|_{\mathcal G}^{-1}(x,y)=-\frac{\mathrm{sr}_1(x,y)}{\mathbf{sres}_1(x,y)}.
\end{equation}
The polynomial $\mathbf{sres}_1(x,y)$ is called the first principal subresultant of $\xi_1,n_1$ and $\xi_2,n_2$. Finally we get the following result. 

\begin{theorem}\label{th-self} Suppose that ${\bf x}$ has no base points lying on ${\mathcal G}$, and let $P\in {\mathcal C}$, $P={\bf x}(t_0,s_0)$, $P\neq P_{\pm \infty}$. If $P$ is a self-intersection, then $(t_0,s_0)$ is a solution of the bivariate polynomial system
\begin{equation}\label{fund-system}
\mathbf{sres}_1(x(t,s),y(t,s))=0,\mbox{ }g(t,s)=0.
\end{equation}
\end{theorem}

The next result shows that, in fact, {\it all} the singularities of ${\mathcal C}$, i.e. the local singularities and the self-intersections, except perhaps for $P_{\pm \infty}$, are solutions of Eq. \eqref{fund-system}. The proof of this result in given in Appendix I, so as not to stop the flow of the paper. 

\begin{proposition} \label{othersing}
Let $(t_0,s_0)\in\mathcal{G}$ be a point such that 
\[(x_0,y_0)=(x(t_0,s_0),y(t_0,s_0))\in {\mathcal C}\]
is not a self-intersection, with
\begin{equation}\label{lacondi}
x_t(t_0,s_0)g_s(t_0,s_0)-x_s(t_0,s_0)g_t(t_0,s_0) = y_t(t_0,s_0)g_s(t_0,s_0)-y_s(t_0,s_0)g_t(t_0,s_0)=0.
\end{equation}
Then $\mathbf{sres}_1(x_0,y_0)=0$.
\end{proposition}

Proposition \ref{othersing} provides the following result.

\begin{theorem}\label{all-the-sings}
Suppose that ${\bf x}$ has no base points lying on ${\mathcal G}$. Then every singularity of ${\mathcal C}$, except perhaps for $P_{\pm \infty}$, is a solution of Eq. \eqref{fund-system}.
\end{theorem}

The analysis of $P_{\pm \infty}$ is postponed to Section \ref{subsec-nueva}. Additionally, there is another point missing in the discussion before. In order for the subresultant chain of $\xi_1,\xi_2$ not to vanish completely, we must require that $\xi_1,\xi_2$ do not share any factor depending on $t$. We identify the cases when this happens in the following two results. The proofs of these results are given in Appendix I. 

\begin{lemma} \label{t-factor}
The polynomials $\xi_1(x,t)$ and $\xi_2(x,y,t)$ have a common factor $t-t_0$ iff $t_0$ corresponds to a base point of ${\bf x}$, lying on ${\mathcal G}$. 
\end{lemma}

\begin{lemma} \label{tx-factor}
The polynomials $\xi_1(x,t)$ and $\xi_2(x,y,t)$ have a common factor $\eta(x,t)$ depending on both $x,t$ iff $x(t,s)$ depends only on $t$.
\end{lemma}

In the case of Lemma \ref{t-factor}, if ${\bf x}$ has some base point lying on ${\mathcal G}$ we remove the common factor depending on $t$, and perform the procedure presented before. In the case of Lemma \ref{tx-factor}, we replace $\xi_2(x,y,t)$ by 
\[\tilde{\xi}_2(y,t)=\mbox{square-free part of }\mbox{Res}_s(\mbox{num}(y-y(t,s)),g(t,s)),\]
and proceed as before.

\subsection{Behavior of ${\mathcal C}$ around the base points of ${\bf x}|_{\mathcal G}$.}\label{analysis-base-points}

Let $Q=(t_0,s_0)\in\mathcal{G}$ be a base point of ${\bf x}|_{\mathcal G}$. Notice that by Lemma \ref{t-factor}, $t=t_0$ must be a root of the content of $\xi_1,\xi_2$ with respect to $t$, and therefore has been previously determined. In this case, ${\bf x}(t_0,s_0)=\left(\frac{0}{0},\frac{0}{0}\right)$. Although the fact that the ${\mathcal G}$ does not have affine singularities guarantees that ${\bf x}(t_0,s_0)$ is defined as a projective point (see Theorem 1.2 of \cite{Shafarevich}), we still need to determine the behavior of ${\bf x}$ when the point $(t_0,s_0)$ is approached; in particular, we need to check if we get an affine point or a point at infinity, in which case we get an infinite branch of ${\mathcal C}$. In order to do this, we distinguish two situations: 

\begin{itemize}
\item [(i)] {\it The point $(t_0,s_0)$ is not a critical point of ${\mathcal G}$:} in this case, by the Implicit Function Theorem $s^2-p(t)=0$ implicitly defines $s=s(t)$ at $t=t_0$. In fact, we can easily find the Taylor expansion of the function $s(t)$ at $t=t_0$, and then study the limits
\[
\mbox{lim}_{t\to t_0}x(t,s(t)),\, \mbox{lim}_{t\to t_0}y(t,s(t)).
\]
If both limits are finite, then $(t_0,s_0)$ generates an affine point of ${\mathcal C}$. Otherwise we have a branch going to infinity, which is an asymptote of ${\mathcal C}$ whenever one of the above limits is finite. 
\item [(ii)] {\it The point $(t_0,s_0)$ is a critical point of ${\mathcal G}$:} in this case $t_0$ is a root of $p(t)$, so $s_0=0$. Now we consider $s=\pm \sqrt{p(t)}$ and we study each branch $s=\sqrt{p(t)}$ and $s=-\sqrt{p(t)}$ separately. We address in more detail the case $s=\sqrt{p(t)}$; for $s=-\sqrt{p(t)}$ the analysis is similar. Now if $s=\sqrt{p(t)}$, for the component $x(t,s)$ we have 
\[
x\left(t,\sqrt{p(t)}\right)=\frac{a_{11}(t)+\sqrt{p(t)}a_{12}(t)}{b_{11}(t)+\sqrt{p(t)}b_{12}(t)}.
\]
We are interested in analyzing the behavior of this function when $t\to t_0$. Since $(t_0,0)$ is a base point of $x(t,s)$, $a_{11}(t_0)=b_{11}(t_0)=0$. Additionally, since $a_{11}(t)$, $a_{12}(t)$, $b_{11}(t)$, $b_{12}(t)$ are relatively prime, it cannot be $a_{12}(t_0)=0$ and $b_{12}(t_0)=0$ simultaneously. Furthermore, $t=t_0$ is a root of $p(t)$, and since $p(t)$ does not have multiple roots, the multiplicity of $t_0$ is 1. Hence we can factor out $(t-t_0)^{1/2}$ in the numerator and denominator of $x(t,\sqrt{p(t)})$, and we get 
\[
x\left(t,\sqrt{p(t)}\right)=\frac{\tilde{a}_{11}(t)+\sqrt{\tilde{p}(t)}a_{12}(t)}{\tilde{b}_{11}(t)+\sqrt{\tilde{p}(t)}b_{12}(t)},
\]
where $\tilde{a}_{11}(t)=\dfrac{a_{11}(t)}{(t-t_0)^{1/2}}$, $\tilde{b}_{11}(t)=\dfrac{b_{11}(t)}{(t-t_0)^{1/2}}$, and $\tilde{p}(t)=\dfrac{p(t)}{t-t_0}$. 

Observe that since $a_{11}(t_0)=b_{11}(t_0)=0$ and $a_{11}(t),b_{11}(t)$ are polynomials, $\tilde{a}_{11}(t_0)=\tilde{b}_{11}(t_0)=0$. Therefore, when $t\to t_0$ the limit of the function $x(t,\sqrt{p(t)})$ is equal to the limit of $a_{12}(t)/b_{12}(t)$ when $t\to t_0$. Since not both $a_{12}(t_0),b_{12}(t_0)$ are zero, the limit is defined whenever $b_{12}(t_0)\neq 0$, and is infinite (in which case we have a branch at infinity) whenever $b_{12}(t_0)=0$. Similarly for the component $y(t,s)$, and for $s=-\sqrt{p(t)}$.   
\end{itemize}

Notice that these ideas can be also used at points $(t_0,s_0)$ where only one
component of ${\bf x}|_{\mathcal G}(t,s)$ is undefined.

\subsection{Computation and study of $P_{\pm \infty}$.}\label{subsec-nueva}

The point at infinity of the curve ${\mathcal G}$ is the center of two {\it places}, i.e. two branches of ${\mathcal G}$. In turn, these two branches generate two branches of ${\mathcal C}$ via ${\bf x}$, which can be centered at affine points or points at infinity denoted by $P_{\pm \infty}$. In order to compute whether or not the $P_{\pm \infty}$ are affine, we must study the (four) limits
\begin{equation}\label{limits}
\begin{array}{cc}
\mbox { }\mbox{lim}_{t\to \pm\infty} {\bf x}\left(t,\sqrt{p(t)}\right), & \mbox { }\mbox{lim}_{t\to \pm\infty} {\bf x}\left(t,-\sqrt{p(t)}\right).
\end{array}
\end{equation}
Notice that we can have at most two different finite values in these limits, corresponding to the case when all $P_{\pm \infty}$ are affine. In order to compute these limits, after performing elementary calculations we arrive to an expression $\frac{\mu_1(t)}{\mu_2(t)}$ where one of the $\mu_i(t)$ is a polynomial, and the other $\mu_i(t)$ involves polynomials and one radical term. Then the limit can be evaluated by just comparing the degrees of the numerator and the denominator; notice that the degree can be a non-integer, rational number in the case of the numerator or denominator involving a square-root. In our experimentation we have checked that a computer algebra system like Maple 18 perfectly computes these limits in almost no time.

It can happen that all $P_{\pm \infty}$, only some of them, or none of them, is affine. If all $P_{\pm \infty}$ are affine and equal, then $P_{\pm \infty}$ is a self-intersection of ${\mathcal C}$. In this case, if the branches at infinity of ${\mathcal G}$ are real, then there are at least two real branches of ${\mathcal C}$ passing through $P_{\pm \infty}$; if the branches are complex and $P_{\pm \infty}$ is real, then $P_{\pm \infty}$ is an isolated point of ${\mathcal C}$. If some $P_{\pm \infty}$ is affine, it can also be a self-intersection of ${\mathcal C}$ when there exists an affine point of ${\mathcal G}$ whose image under ${\bf x}(t,s)$ coincides with this $P_{\pm \infty}$. This can be checked by solving the bivariate system $\{{\bf x}(t,s)=P_{\pm \infty},\mbox{ }g(t,s)=0\}$.

Additionally, when some of the $P_{\pm \infty}$ are affine, we can check whether they are local singularities by checking whether the limit for $t\to \pm\infty$ of the derivative of ${\bf x}(t,\pm \sqrt{p(t)})$ vanishes.

\subsection{Construction of $G_{\mathcal C}$.} \label{top-final}

Let $Q_1=(t_1,s_1),\ldots,Q_r=(t_r,s_r)$ be the points of ${\mathcal G}$ computed in (i)-(iv). Since the $Q_i$ belong to ${\mathcal G}$ and the graph associated with ${\mathcal G}$ can be computed by means of well-known methods \cite{Berb, CLPPRT09, EKW07, GVN}, we know how to connect the $Q_i$ to each other. Furthermore, from the preceding sections the behavior of ${\bf x}$ around the $Q_i$ is clear. Now the vertices of $G_{\mathcal C}$ are the images $P_i={\bf x}(Q_i)$, whenever ${\bf x}(Q_i)$ (or the limit of ${\bf x}(t,s)$ as $(t,s)\to Q_i$, in the case of base points) is defined, and we connect two of these vertices iff their preimages $Q_i$ are connected to each other in $G_{\mathcal G}$. Furthermore, we also include as vertices of $G_{\mathcal C}$ the points $P_{\pm \infty}\in{\mathcal C}$ coming from the point at infinity of ${\mathcal G}$, in case they are affine. 

Additionally, the graph associated with ${\mathcal G}$ can have open edges (representing branches tending to infinity), corresponding to the edges of ${\mathcal G}$ with some vertex where some component of ${\bf x}$ becomes infinite, or branches of ${\mathcal G}$ tending to infinity, in the case when some $P_{\pm \infty}$ is at infinity. Also, we must check that the edges of the graph associated with ${\mathcal C}$ do not intersect except at the self-intersections of ${\mathcal C}$. This amounts to computing the intersection of two segments, which is straightforward and negligible in terms of computation time. If two edges of the graph associated with ${\mathcal C}$ intersect at a point which is not a self-intersection of ${\mathcal C}$ (notice that the computation of the self-intersections of ${\mathcal C}$ was addressed in the previous subsections) we just deform slightly one of the edges, or introduce additional vertices so that the intersection is avoided.

\begin{theorem} \label{all-the-branches}
Let $G_{\mathcal C}$ be the graph associated with ${\mathcal C}$ according to the description in the preceding subsections. Then $G_{\mathcal C}$ and ${\mathcal C}$ are isotopic. 
\end{theorem}

\begin{proof} Once we compute the points of ${\mathcal G}$ where ${\bf x}$ becomes infinite, ${\mathcal G}$ is segmented into finitely many portions $\ell_1,\ldots,\ell_p$ where ${\bf x}$ is continuous. Each $\ell_i$ is connected, and by continuity ${\bf x}(\ell_i)$ is connected as well. Furthermore, by the birationality of ${\bf x}|_{\mathcal G}$ the correspondence between the $\ell_i$ and the ${\bf x}(\ell_i)$ is $1:1$. Since ${\mathcal C}={\bf x}({\mathcal G})$ and ${\bf x}({\mathcal C})$ coincides with the union of the ${\bf x}(\ell_i)$, we just need to show that the graph $G_{\mathcal C}$ is isotopic to the union of the ${\bf x}(\ell_i)$. Since in $G_{\mathcal C}$ we are just deforming each ${\bf x}(\ell_i)$ into a segment, in order to show that $G_{\mathcal C}$ and ${\mathcal C}$ are isotopic we just need to show that no self-intersections of ${\mathcal C}$ are missed, and that no other self-intersections are introduced. The former is guaranteed by construction, since in the process of computing $G_{\mathcal C}$ all the self-intersections of ${\mathcal C}$ are identified. The latter is guaranteed by checking that two edges do not intersect at a point which is not a self-intersection of ${\mathcal C}$. 
\end{proof}

\begin{example} Let 
$$g(t,s)=s^2+t^4-t^3-27t^2+25t+50=0,$$
and let 
\[
{\bf x}(t,s)=(x(t,s),y(t,s))=\left ({\dfrac {{t}^{4}-{t}^{3}+{t}^{2}+5\,s-t}{{t}^{6}+1}},{\dfrac {{t}^{4}+{
t}^{3}-{t}^{2}-5\,s+t}{{t}^{6}+1}}\right).
\]
The curve ${\mathcal C}={\bf x}({\mathcal G})$ is a hyperelliptic curve of genus one.

First we compute the real points $(t,s)\in {\mathcal G}$ generating the vertices of $G_\mathcal{C}$:
\begin{itemize}
\item [(i)] Critical points of $g(t,s)=0$, i.e. points $(t,0)$ with $p(t)=0$: $$Q_1=(-5,0) ,  Q_2=(-1,0) ,  Q_3=(2,0)  \text{ and }Q_4=(5,0).$$	
\item [(ii)] Points of ${\mathcal G}$ giving rise to critical points of ${\mathcal C}$.
Local singularities and ramification points are generated by the points $(t,s)$ solutions of the system
$$
g(t,s)=0, \quad x_tg_s-x_sg_t= 0.
$$
The real solutions (written only with two digits) are:
$$
Q_5=(- 4.98, - 2.05),
Q_6=(- 3.21, - 13.00),
Q_7=(- 1.16, - 3.47), 
$$
$$
Q_8=(- 1.12,  3.08),
Q_9=( 2.15,  3.11),
Q_{10}=( 2.24, - 3.97),$$ $$
Q_{11}=( 3.76, - 9.54),
Q_{12}=( 4.96, - 2.52).
$$

\bigskip

Now we compute the points of ${\mathcal G}$ giving rise to self-intersections of ${\mathcal C}$. We have:
\begin{eqnarray*}\xi_1(x,t)&=&\left( {t}^{12}+2\,{t}^{6}+1 \right) {x}^{2}+ \left( -2\,{t}^{10}+2\,
{t}^{9}+\cdots \right) x+{t}^{8}+\cdots +1250,
\end{eqnarray*}
and 
$$ \xi_2(x,y,t)= (t^6+1)(x+y)-2t^4.
$$

The self-intersections of ${\mathcal C}$ are generated by the real solutions of the system $\{ \mathbf{sres}_1(x(t,s),y(t,s))=  0,\quad g(t,s)=0\},$ which are
$Q_{13}=( -3.75,   -13.14)$, $Q_{14}=(  -2.32,   -10.61)$, $Q_{15}=( 2.32 ,  -4.62 )$   and  $Q_{16}=(  3.75 ,  9.53  ).$

The points $Q_{13}$ and $Q_{16}$ both generate the same point, $P_{13}$, and the points $Q_{14}$ and $Q_{15}$ both generate the point $P_{14}$ (see Figure \ref{detalles}).
\item [(iii)] Points of $\mathcal{G}$ where some component of ${\bf x}$ is not defined: there are neither base points nor vertical asymptotes.
\item [(iv)] Starting and ending points for open branches of $G$: There are not open branches. In particular, in this case we do not need to analyze the points $P_{\pm \infty}$, since they are either non-real, or real isolated points of ${\mathcal C}$, which we do not consider.
\end{itemize}
Finally, we compute the images $P_i ={\bf x}(Q_i)$, and we connect them according to how the $Q_i$ are connected in ${\mathcal G}$. The graph associated with ${\mathcal G}$ is shown in Fig. \ref{curvasej} (left). The graph associated with ${\mathcal C}$ is also shown in Fig. \ref{curvasej} (right). Additionally, in the graph associated there are several points very close to each other: some details on the topology of ${\mathcal C}$ are given in Fig. \ref{detalles}.

\begin{figure}[h]
$$
\begin{array}{cc}
\includegraphics[scale=0.3]{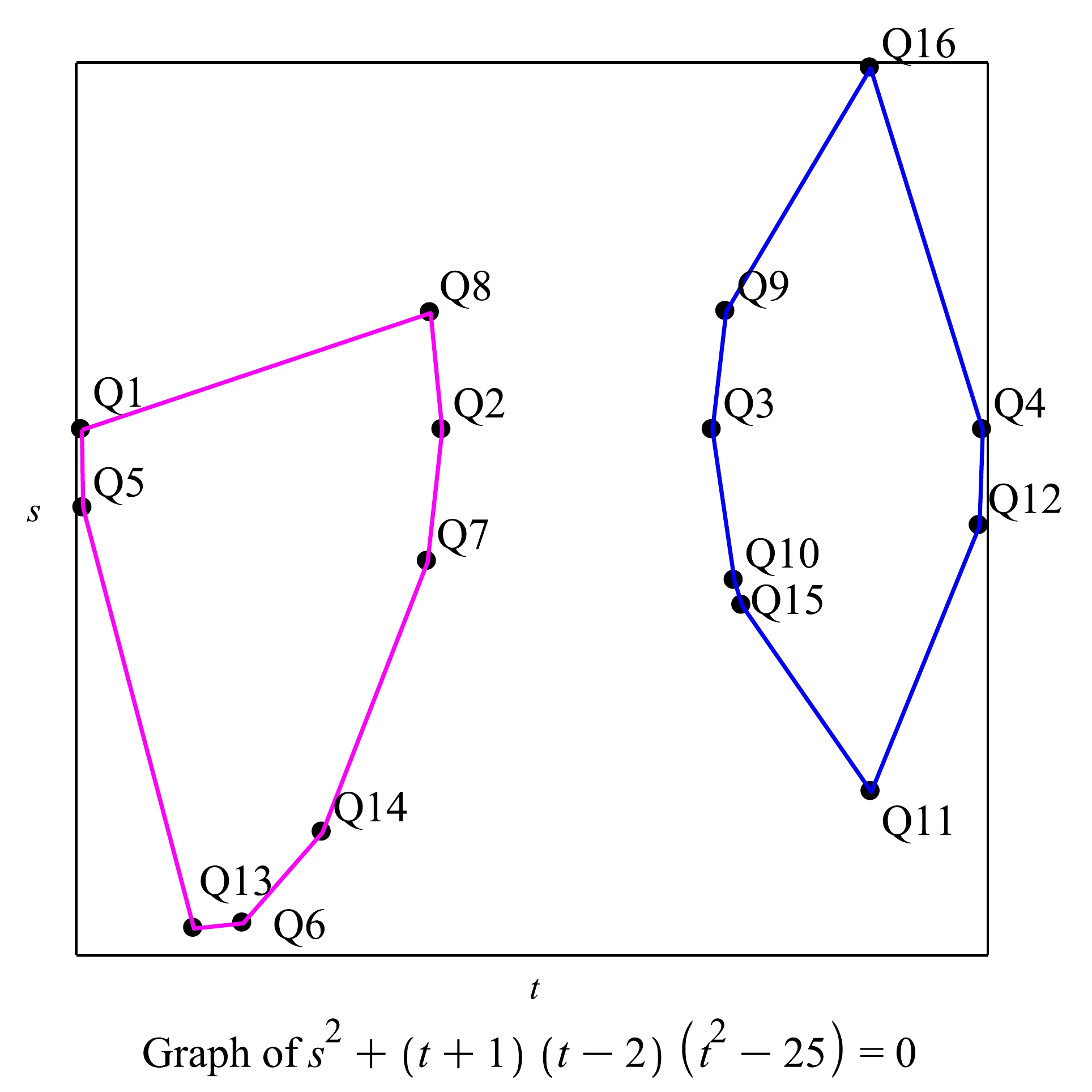}&
\includegraphics[scale=0.3]{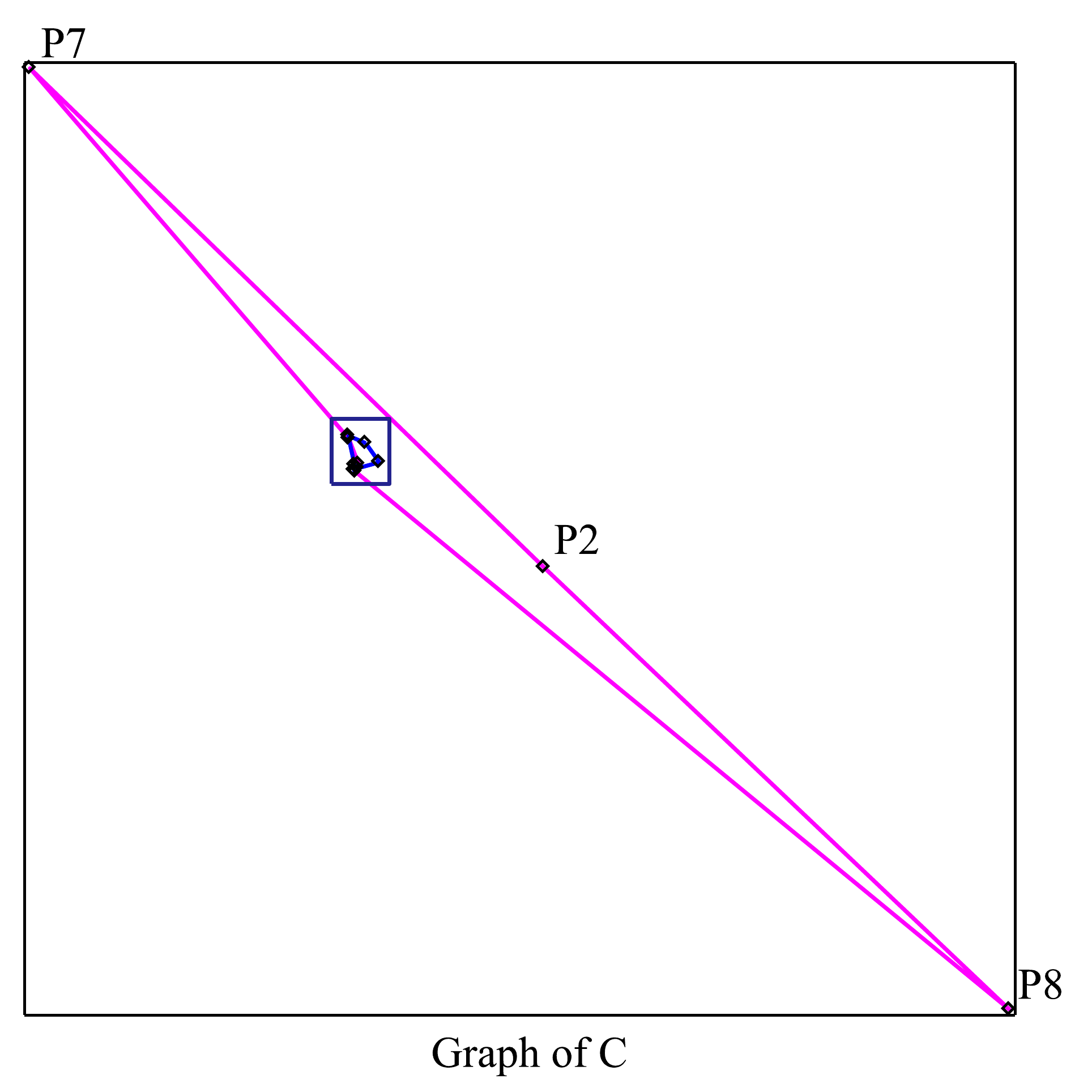}
\end{array}
$$
\caption{Correspondence between the edges of $G_{\mathcal G}$ and $G_{\mathcal C}$.}\label{curvasej}
\end{figure}
\begin{figure}[h]
$$
\begin{array}{cc}
\includegraphics[scale=0.25]{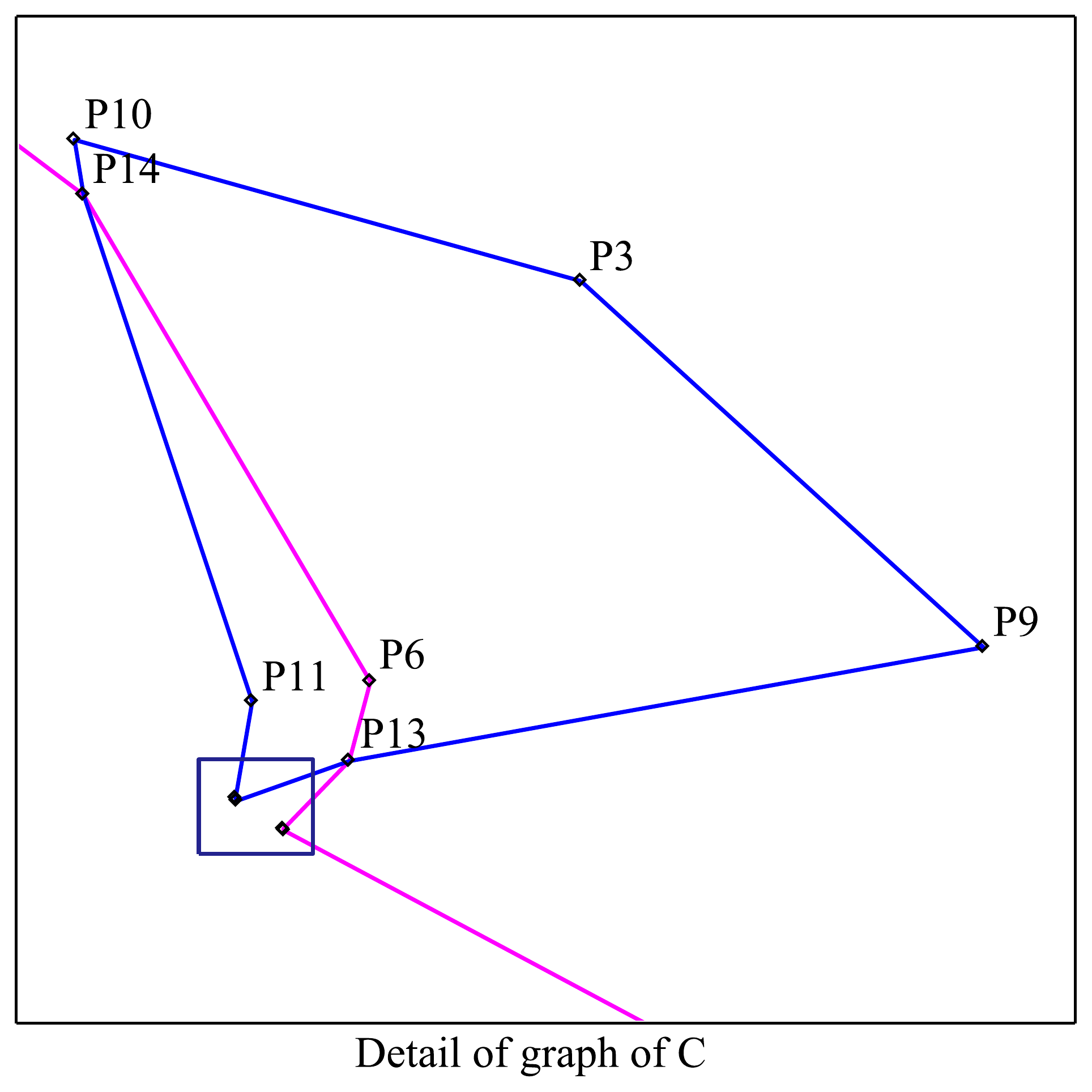}&
\includegraphics[scale=0.25]{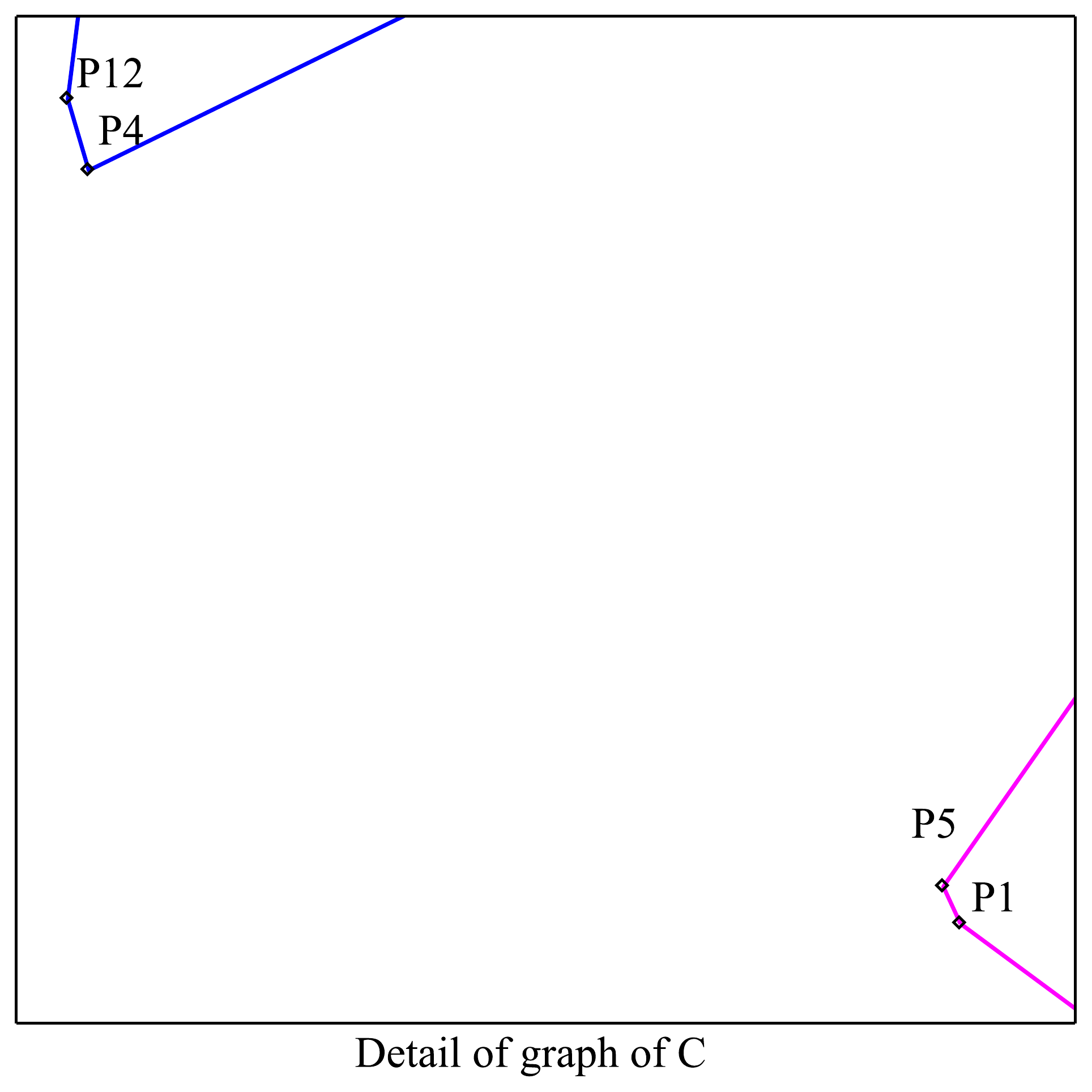}
\end{array}
$$
\caption{Details}\label{detalles}
\end{figure}

\end{example}

\section{The space case. }\label{sec-space}

Here we consider ${\bf x}:{\Bbb R}^2\to {\Bbb R}^3$, where 
\[
{\bf x}(t,s)=(x(t,s),y(t,s),z(t,s))=\left (\frac{A_1(t,s)}{B_1(t,s)},\frac{A_2(t,s)}{B_2(t,s)},\frac{A_3(t,s)}{B_3(t,s)}\right).
\]
We let ${\mathcal C}={\bf x}({\mathcal G})$, where ${\mathcal G}$ is defined by Eq. \eqref{ex-hyper}. In this case, the strategy requires to first birationally project ${\mathcal C}$ onto the $xy$-plane, compute the topology of the projection using the results in Section \ref{sec-planar}, and then lift this projection to get the topology of the curve ${\mathcal C}$. Fig. \ref{isoto} shows why we need to compute a planar projection in order to build a graph isotopic to ${\mathcal C}$. In Fig. \ref{isoto} we have two different curves which are not isotopic: the curve at the left is topologically equivalent to two \emph{entwined} circles; however, the curve at the right is topologically equivalent to two circles which are not entwined. While the projection onto the $xy$-plane of both curves is the same, the different topology in both cases appears when the projection is lifted to space. Notice that if we tried to compute directly the topology of ${\mathcal C}$ in this case using the information on ${\mathcal G}$, we might not be able to distinguish one situation from the other. However, the problem disappears if a projection is used. 

\begin{figure}
$$\begin{array}{c}
  \includegraphics[scale=0.5]{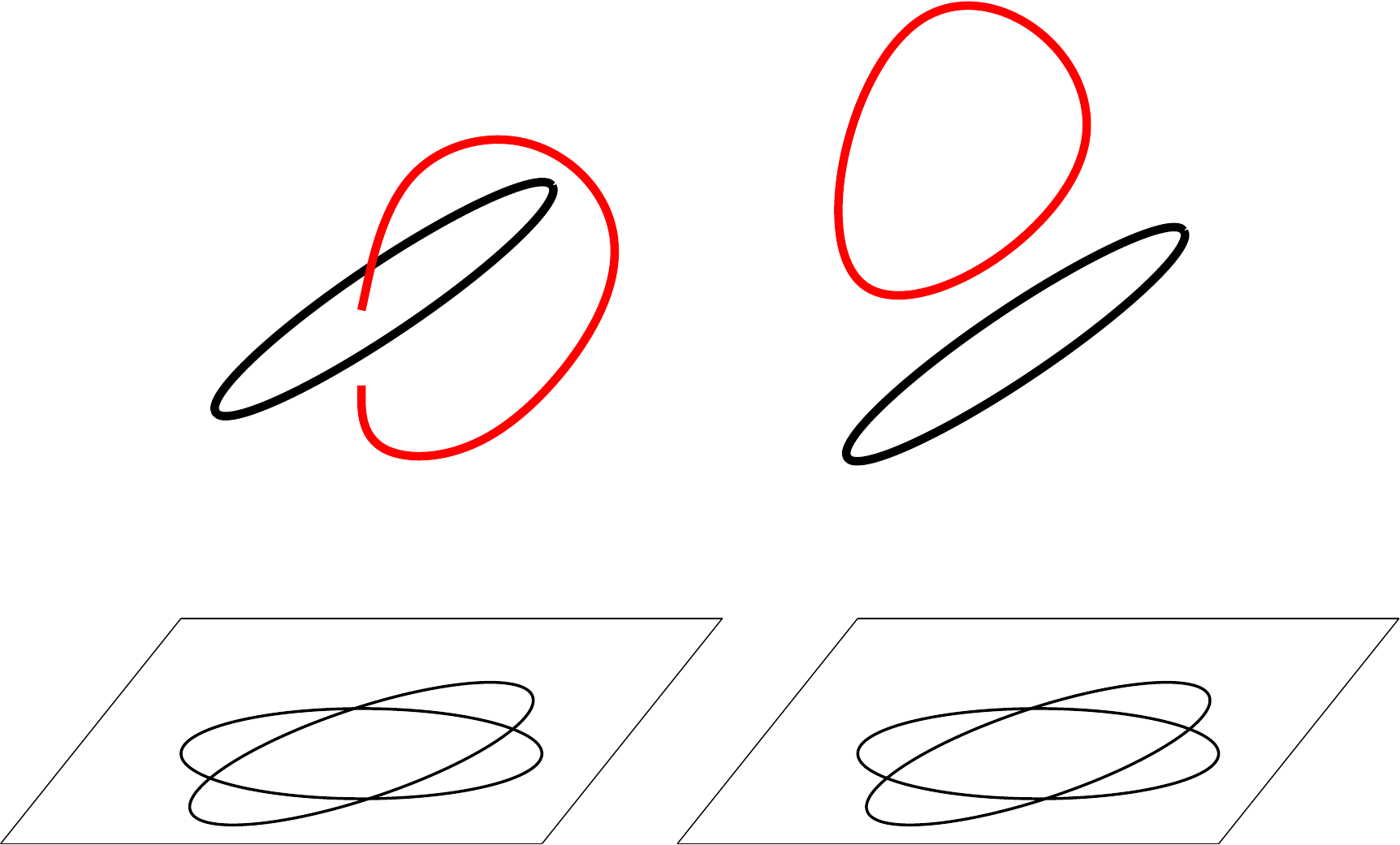}
 \end{array}$$
 \caption{In the space case we need a projection of the curve.}\label{isoto}
\end{figure}

Let ${\mathcal C}^{\star}=\pi_{xy}({\mathcal C})$, where $\pi_{xy}$ denotes the projection onto the $xy$-plane, and let $\tilde{\bf x}=\pi_{xy}\circ {\bf x}$. Fig. \ref{fig-relat} illustrates the relationship between ${\mathcal G}$, ${\mathcal C}$ and ${\mathcal C}^{\star}$. We need two hypotheses this time: 

\begin{itemize}
\item [(H1)] The restriction $\tilde{\bf x}|_{\mathcal G}$ is birational.
\item [(H2)] The curve ${\mathcal C}^{\star}$ does not have any asymptotes parallel to either the $y$-axis, or the $z$-axis. 
\end{itemize}

\noindent It is also customary, when computing the topology of a space curve ${\mathcal C}$, to require that ${\mathcal C}$ has no component parallel to the $z$-axis. However, in our case ${\mathcal C}$ is irreducible, i.e. ${\mathcal C}$ consists of only one component. If ${\mathcal C}$ reduces to a line parallel to the $z$-axis, then the only possibility is that both $x(t,s),y(t,s)$ are constant, which is a trivial case.

Hypothesis (H1) implies that ${\bf x}$ itself is birational when restricted to ${\mathcal G}$, and that $\pi_{xy}$ is also birational when restricted to ${\mathcal C}$; in turn, this means that there are not two different branches of ${\mathcal C}$ projecting as a same branch of ${\mathcal C}^{\star}$, and therefore that the branches of ${\mathcal C}$ are the result of lifting to space the branches of the projection ${\mathcal C}^{\star}=\pi_{xy}({\mathcal C})$. Hypothesis (H1) can be checked, as observed in Section \ref{sec-planar}, by taking a random point $(t_0,s_0)\in {\mathcal G}$ and determining the preimages of $\tilde{\bf x}(t_0,s_0)$. Hypothesis (H2) can be checked by testing whether or not $B_2(t,s)=g(t,s)=0$ has some solution where $A_2(t,s)\cdot B_1(t,s)\neq 0$, and whether or not $A_2(t,s)=g(t,s)=0$ has some solution where $A_1(t,s)\cdot B_2(t,s)\neq 0$. Both hypotheses, (H1) and (H2), guarantee that: (i) the topology of ${\mathcal C}^{\star}$ could be computed by applying the ideas in Section \ref{sec-planar}; (ii) the topology of ${\mathcal C}$ could be computed from the topology of ${\mathcal C}^{\star}$, by lifting a (planar) graph isotopic to ${\mathcal C}^{\star}$. In our case, however, we do not need to compute first the topology of ${\mathcal C}^{\star}$; instead, as in Section \ref{sec-planar}, we determine all the points $(t,s)\in {\mathcal G}$ giving rise to ``notable" points of ${\mathcal C}$, and incorporate those points as vertices of $G_{\mathcal G}$. Then the edges of $G_{\mathcal G}$ are mapped onto edges of $G_{\mathcal C}$ as we did in Section \ref{sec-planar}. 

Hypotheses (H1) and (H2) can always be achieved when ${\bf x}|_{\mathcal G}$ is birational. Indeed, under this assumption, for almost all random affine changes of coordinates $\phi$ and renaming ${\bf x}:={\bf x}\circ \phi$, $\pi_{xy}|_{\mathcal C}$ is birational, i.e. two different branches of ${\mathcal C}$ do not project as a same branch of ${\mathcal C}^{\star}$. As a consequence $\tilde{\bf x}|_{\mathcal G}$ must be birational. 

\begin{figure}
$$\begin{array}{c}
  \includegraphics[scale=0.6]{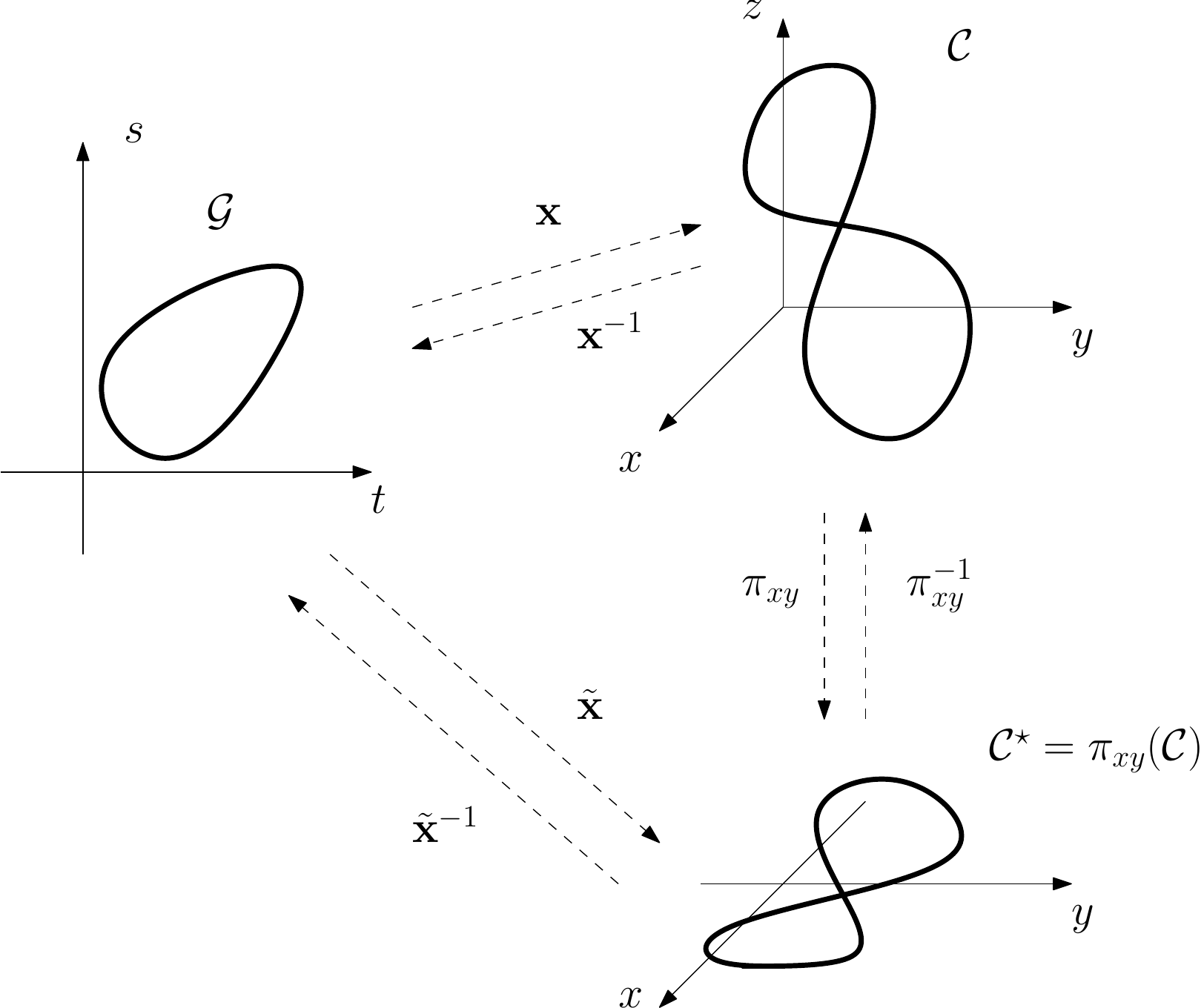}
 \end{array}$$
 \caption{Relationship between the curves ${\mathcal G}$, ${\mathcal C}$, and ${\mathcal C}^{\star}$.}\label{fig-relat}
\end{figure}

\vspace{0.3 cm}
In this case, we need to include the following points as vertices of $G_{\mathcal G}$:

\begin{itemize}
\item [(i)] \emph{Critical points of $g(t,s)=0$}, i.e. points of ${\mathcal G}$ where $g_s=0$. 
\item [(ii)] \emph{Points of ${\mathcal G}$ giving rise to critical points of ${\mathcal C}^{\star}$}.
\item [(iii)] \emph{Points of $\mathcal{G}$ where some component of ${\bf x}$ is not defined}.
\item [(iv)] \emph{Starting and ending points for open branches of ${\mathcal G}$}.
\end{itemize}

The points in (i), (ii), (iii) are computed as in Section \ref{sec-planar}; observe that the pairs $(t,s)$ generating singularities and points of ${\mathcal C}$ with tangent parallel to the $z$-axis are among the critical points of ${\mathcal C}^{\star}$ (see \cite{AS05, ADT10}). Once the points $Q_i=(t_i,s_i)$, $i=1,\ldots,r$ in (i)-(iv) are computed, we can find, whenever they are defined, the images $P_i={\bf x}(Q_i)$ or the limit points and proceed as in Section \ref{sec-planar} in order to connect the $Q_i$. 

\section{Experimentation.} \label{sec-exp}

In this section we report on the experimentation carried out in the case of both 2D and 3D curves. The algorithms have been implemented in \texttt{Maple 2017}, and the examples
run on an Intel Core i3 processor with speeds revving up to 3.06 GHz. 

Next, we first present examples of the 2D algorithm. In Table 1, we include for each curve, the genus, the total degree 
($d_i$) and the number of terms of the implicit equation (n.terms), the timings in seconds ($t_0$) taken by our algorithm, and the timings in seconds ($t_1$) corresponding to the algorithm in \cite {GVN}, also implemented in \texttt{Maple}, which uses the implicit equation of the curve. Additionally, in Table 1 we checkmark whether each example corresponds to a case where the points $P_{\pm \infty}$ are affine (the column $P_{\pm \infty}$ aff.), and whether the curve has self-intersections (S.I.). The last column provides some extra comments on the existence of base points or asymptotes. The parametrizations corresponding to these examples are given in Appendix II. The graphs corresponding to the examples in Table 1 are shown in Figure (\ref{2d}); from left to right, we have Examples 1, 2, 3 in the first row, 4, 5, 6 in the second row and 7, 8, 9 in the third row.

\vspace{0.4 cm}
\begin{center}
\hspace{-0.8 cm}\begin{tabular}{|c|c|l|c|c|c|c|c|c|} \hline
Example  &genus & $d_i$ & n.terms  &$P_{\pm \infty}$ aff. & S.I. &  $t_0$ & $t_1$ & Obs.   \\ 
\hline \hline 1 &0 & 10 &  57 &\checkmark & \checkmark &  0.310 &  0.270 & Asymptotes   \\ 
\hline 2   &1 & 14 & 81  &  & \checkmark &   0.625 & $^{*}$  & Asymptotes    \\  
\hline 3   & 2& 6 & 26  & \checkmark & \checkmark &  0.398 &  0.110  &  \\
\hline 4   &1 & 12 & 81  & \checkmark & \checkmark &  	 0.529 & $^{*}$ & Base points   \\  
\hline 5   &2 & 12 & 75  & \checkmark & \checkmark &   0.543 & $^{*}$   &  \\
\hline 6   &2 & 11 &  75 &  & \checkmark &  0.777 & $^{*}$   &   \\
\hline 7   &2 & 12 &  75 & \checkmark & \checkmark &   0.443 & $^{*}$ &   \\
\hline 8   &1 & 6 & 23  & \checkmark  & \checkmark &   0.484 &  0.108  & \\
\hline 9   & 2& 9 & 55  & \checkmark & \checkmark &   1.069 & 0.308    & \\
\hline
\end{tabular}

\smallskip
{\bf Table 1:} 2D Examples.
\end{center}

\vspace{0.4 cm}

$^{*}$: Computation was cancelled after fifteen minutes. 

Notice that when the algorithm in \cite {GVN} succeeds, it provides better timings than our algorithm. However, in most cases the implicit equation of the curve is too big, and the algorithm in \cite {GVN} gets stuck. 

\begin{figure}[ht]
\begin{center}
\centerline{$\begin{array}{ccc}
\includegraphics[scale=0.25]{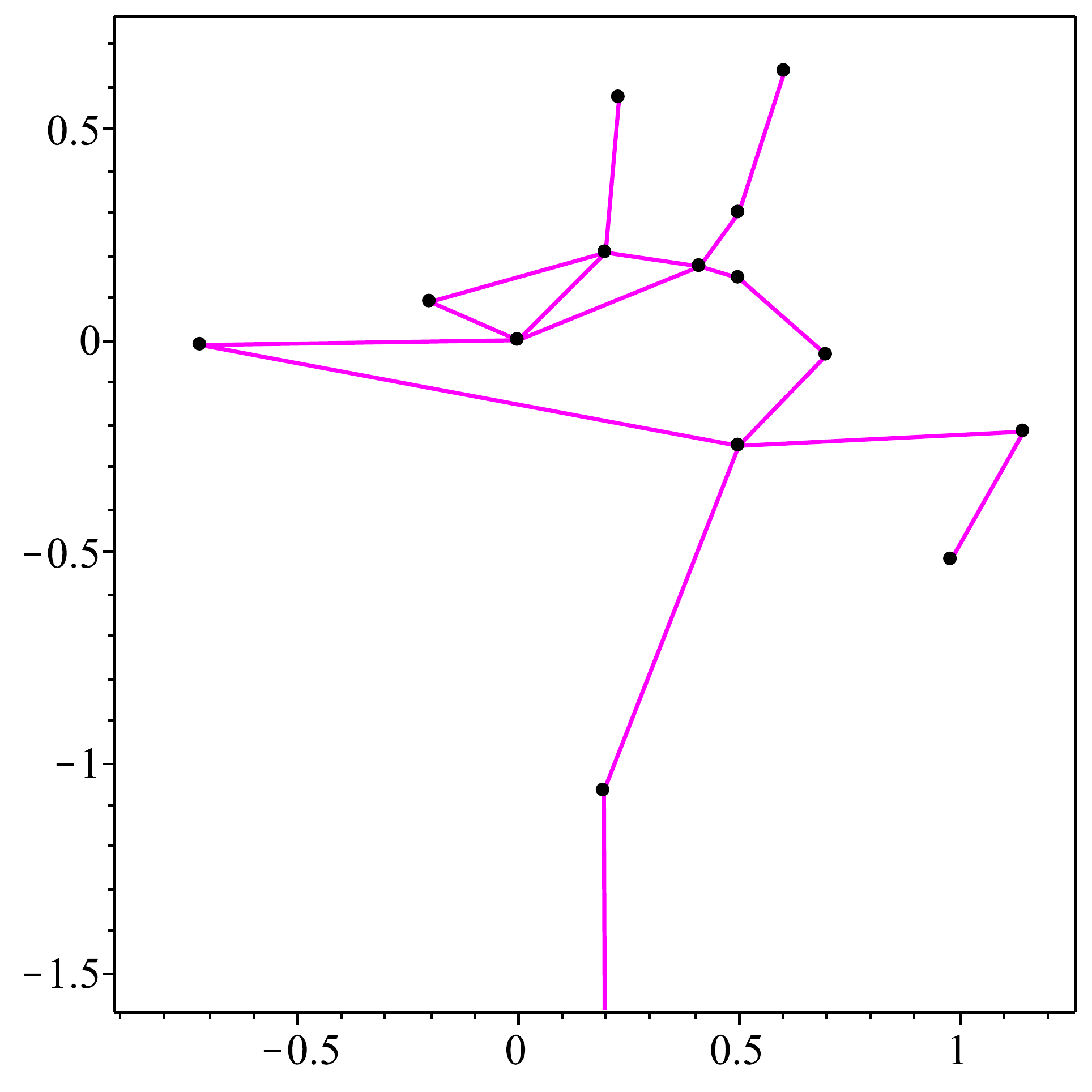} &
\includegraphics[scale=0.25]{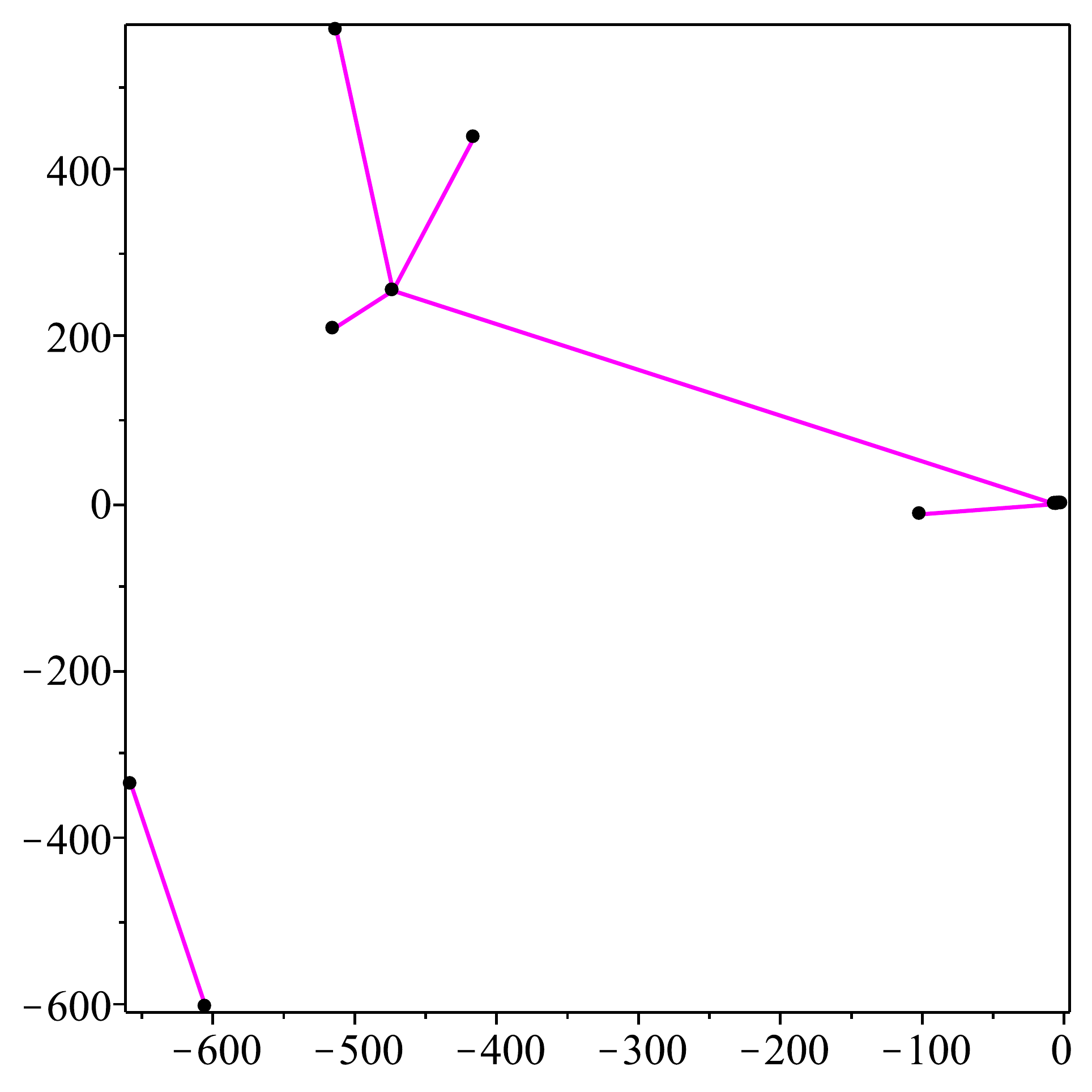} &
\includegraphics[scale=0.25]{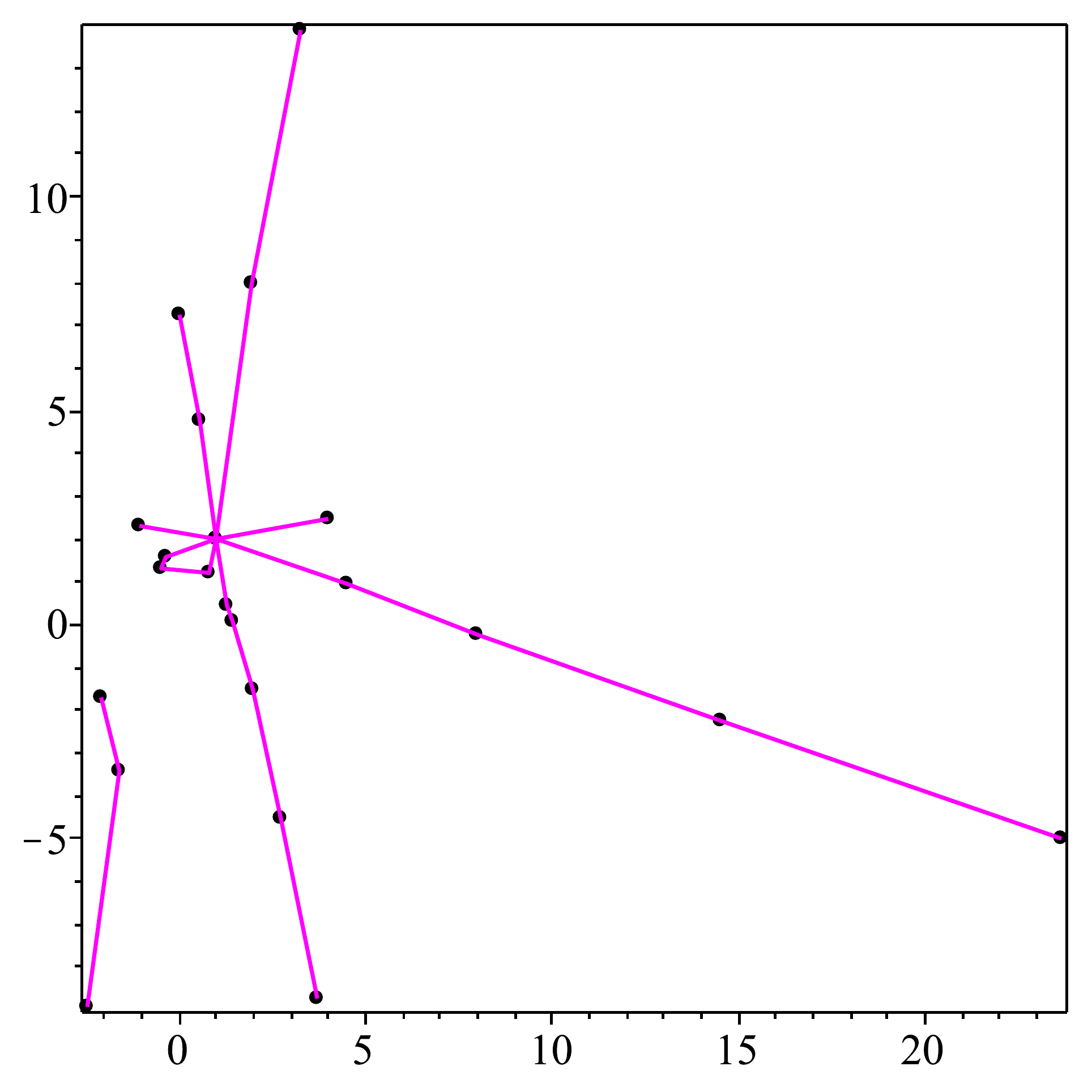} \\
\includegraphics[scale=0.25]{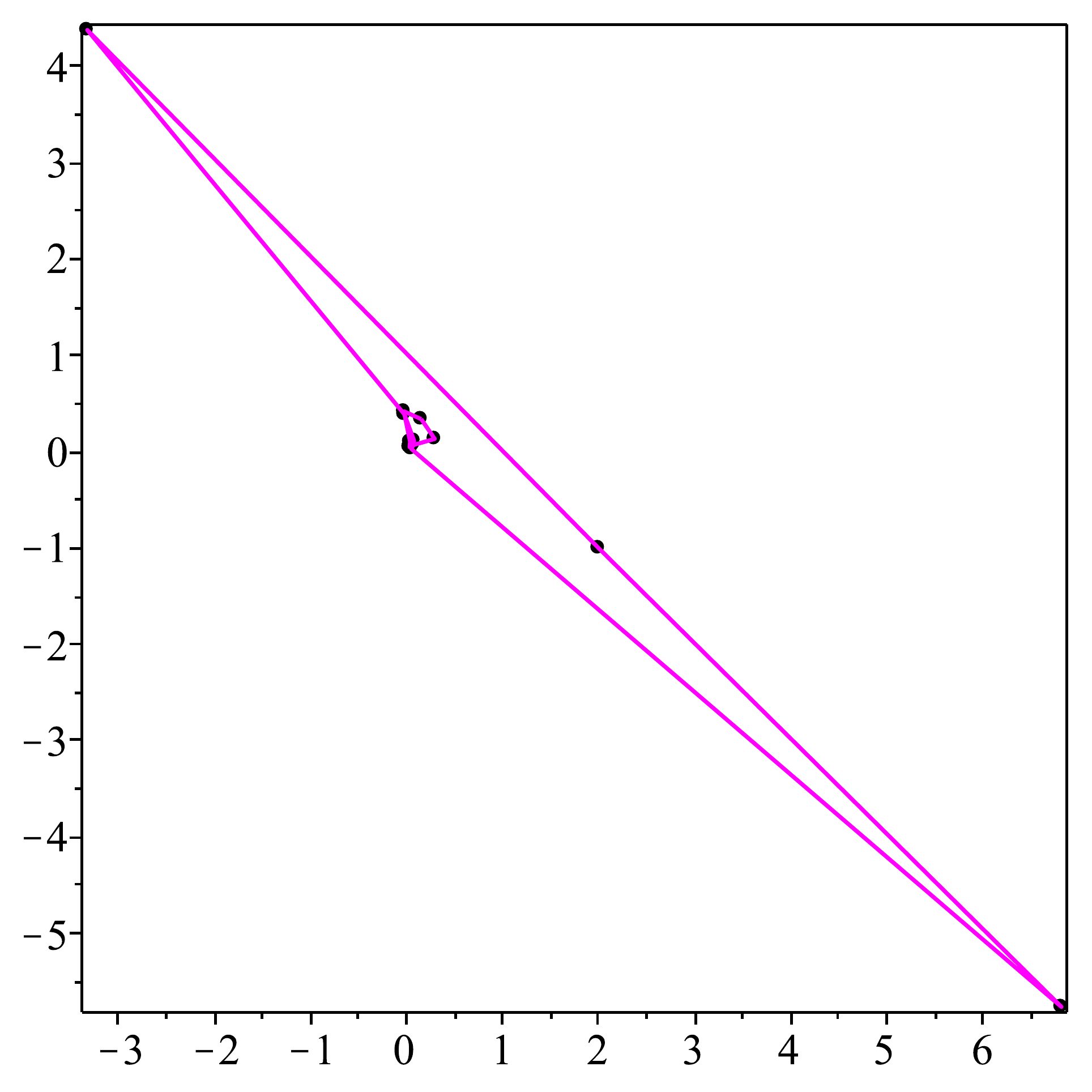} &
\includegraphics[scale=0.25]{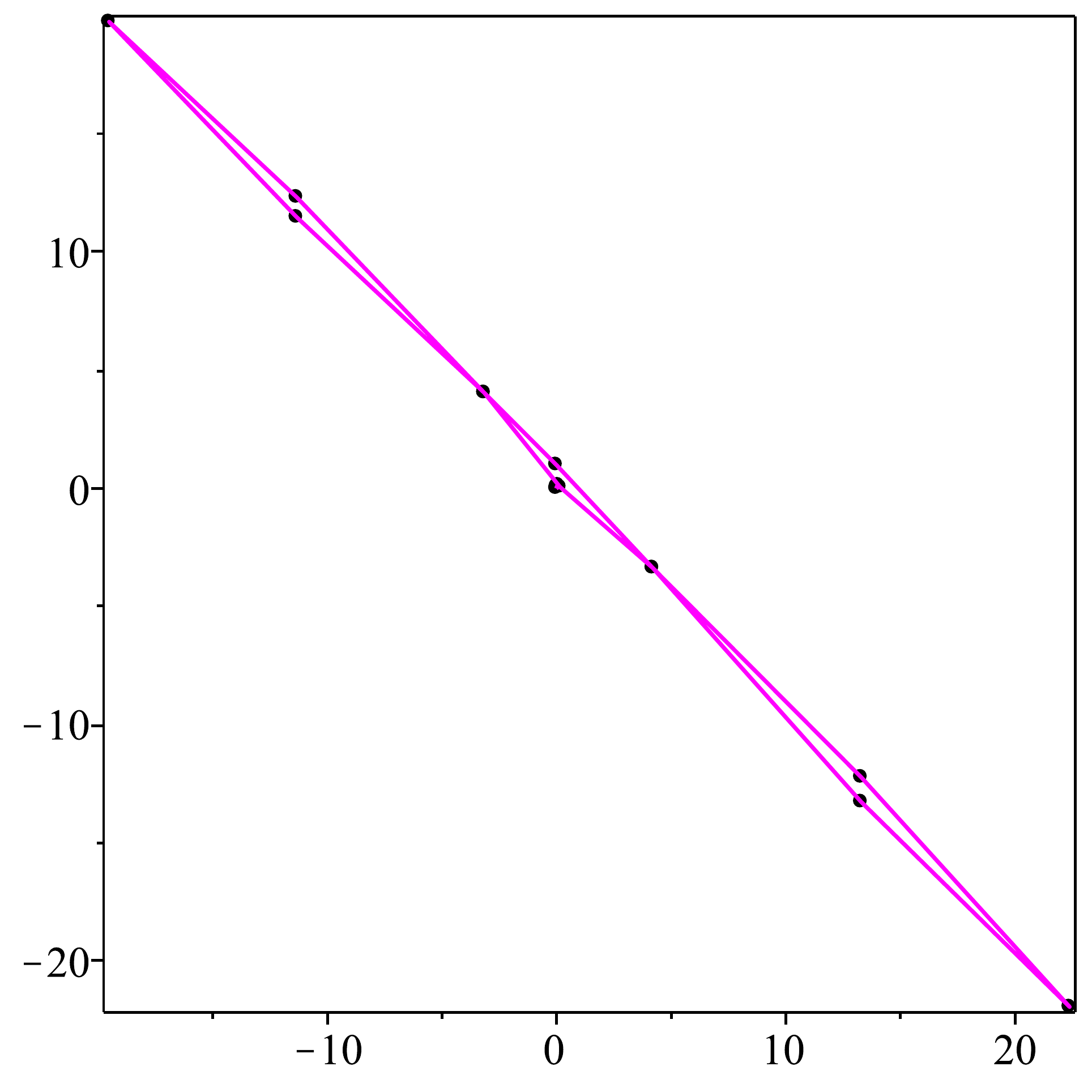} &
\includegraphics[scale=0.25]{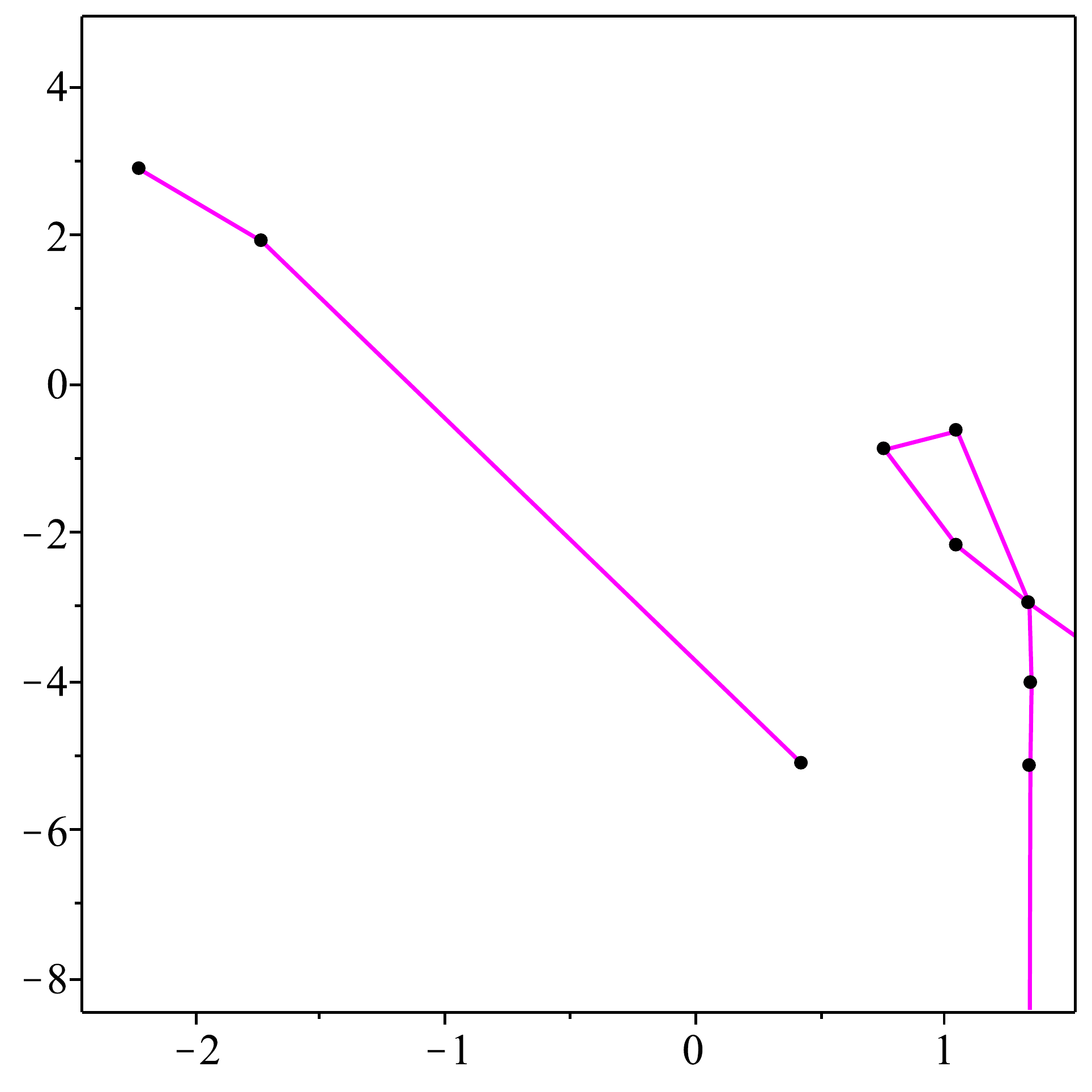}\\
\includegraphics[scale=0.25]{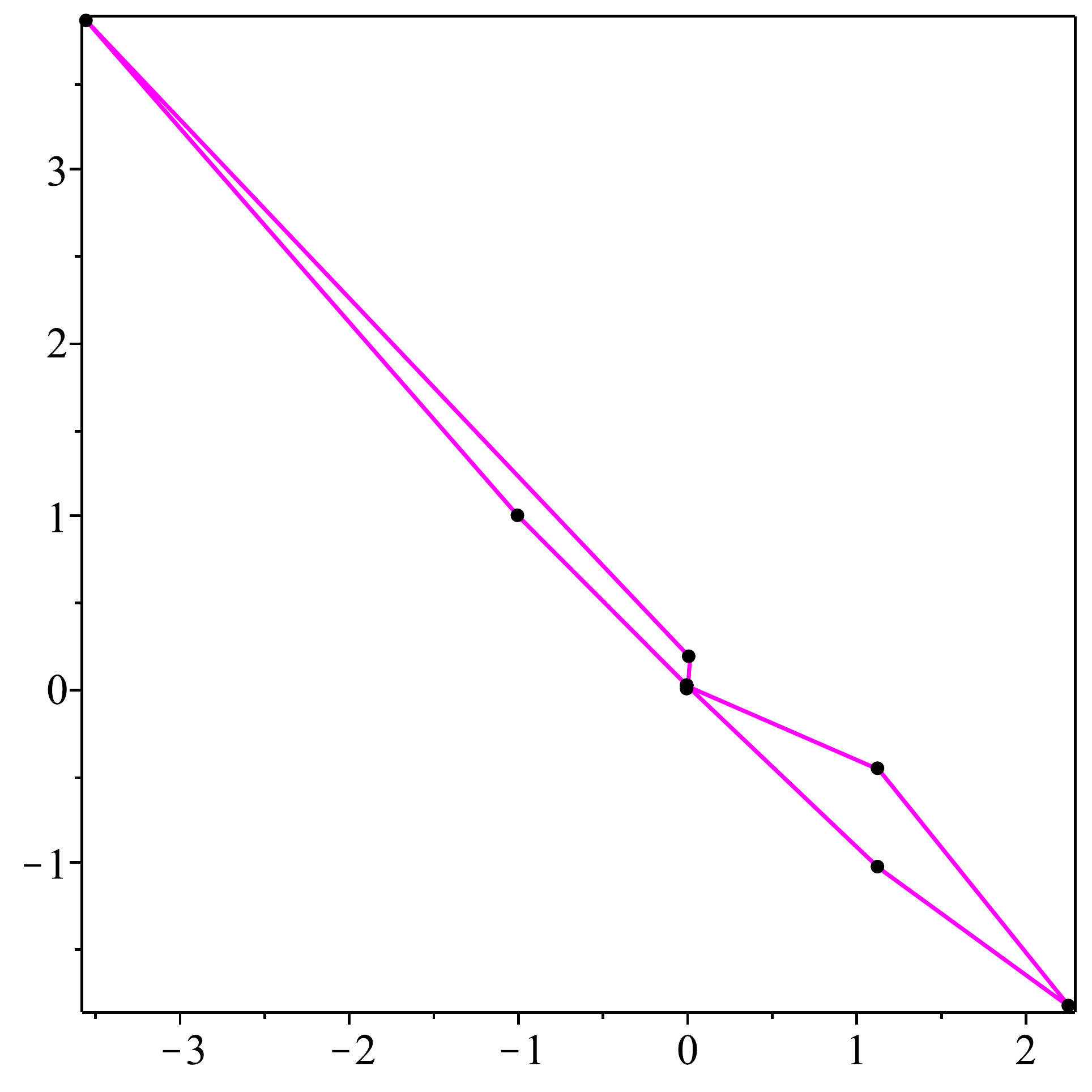} &
\includegraphics[scale=0.25]{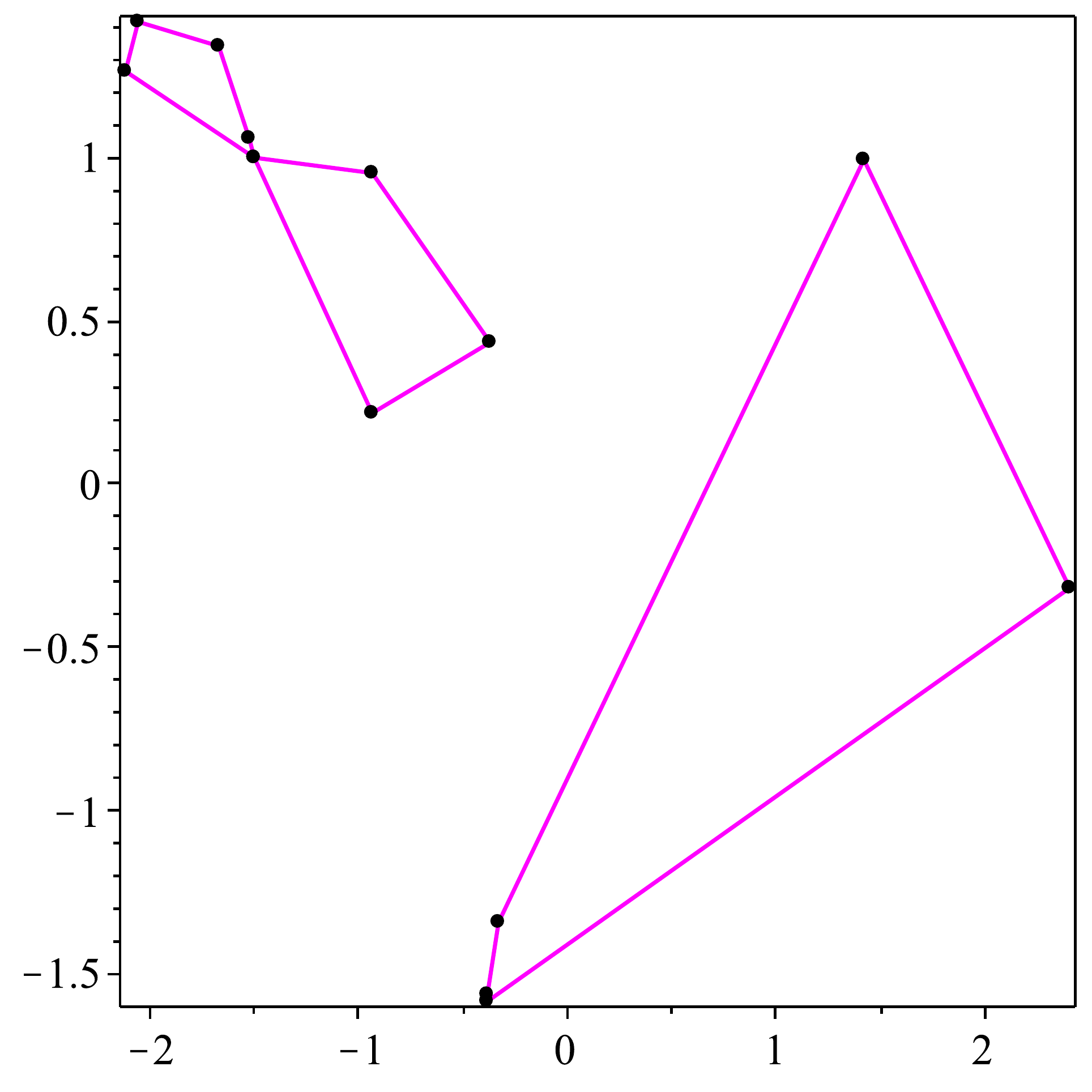} &
\includegraphics[scale=0.25]{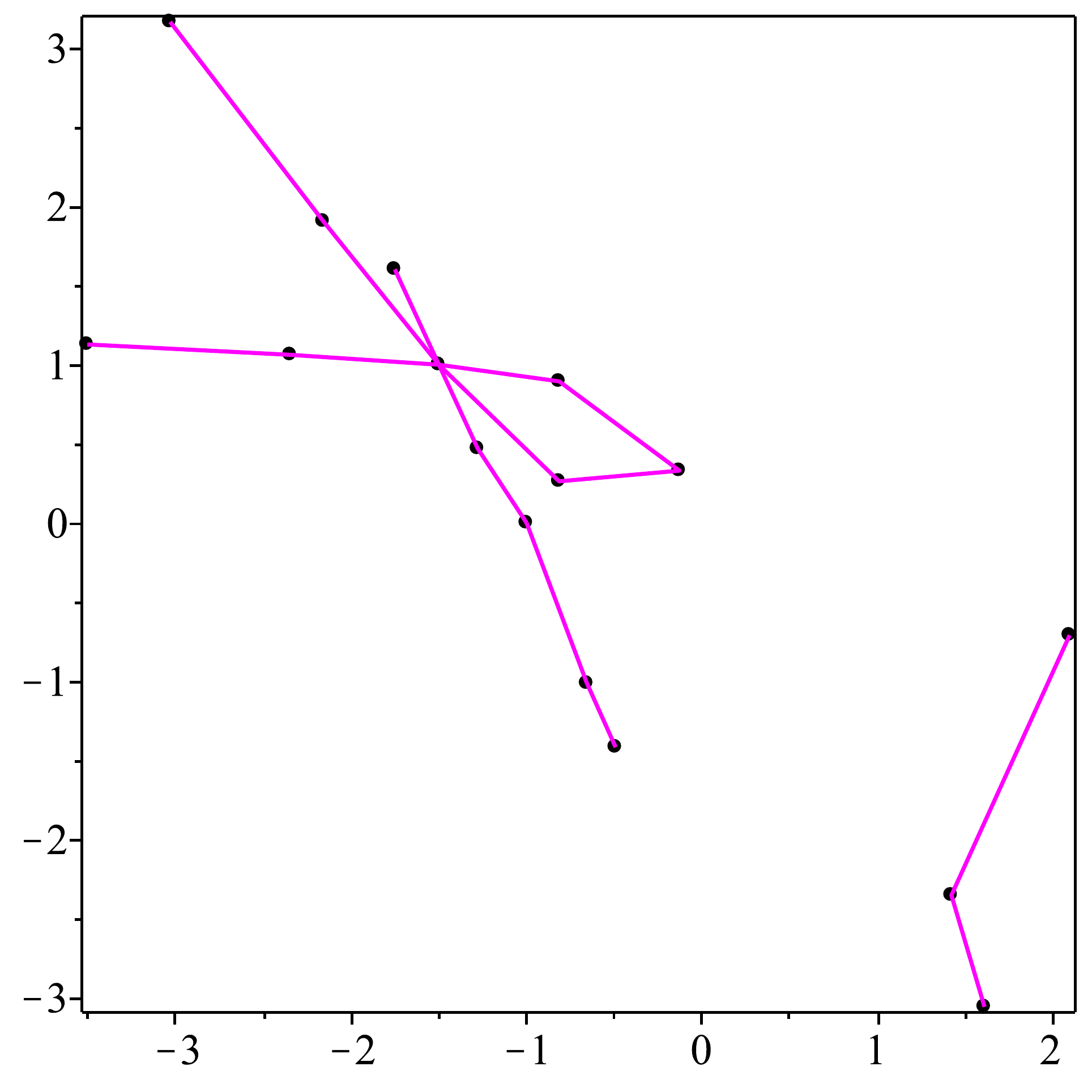} 
\end{array}$}
\end{center}
\caption{Examples of the 2D algorithm.}\label{2d}
\end{figure}

\medskip
Finally, we present examples of the 3D algorithm. In Table 2, for each curve we include the genus, the total degree 
($d_i$) and the number of terms of the implicit equation
of the projection onto the $xy$-plane (n.terms), and the timing in seconds taken by our algorithm ($t_0$); the parametrizations corresponding to each curve are given in Appendix III. Additionally, we include two columns on the nature of $P_{\pm \infty}$ and the existence of self-intersections, as in Table 1. In the last column we include some observations on how we generated the example, in some interesting cases. 

\vspace{0.4 cm}
\begin{center}
\hspace{-0.8 cm}\begin{tabular}{|c|c|l|c|c|c|c|c|c|} \hline
Example  &genus & $d_i$ & n.terms  &$P_{\pm \infty}$ aff. & S.I. &  $t_0$ & Obs.    \\ 
\hline \hline 1 &  4 &  10 &  66  &   &  & 1.543 & \\ 
\hline 2   &  2 &  6 & 16   &   \checkmark &  &  0.344  &  Int. con. and quadric \\
\hline 3   & 7 & 16 &  153  &    & \checkmark &   78.252 &  Int. ruled and quadric  \\
\hline 4   &3  & 8  &  42  &    &  &   0.537 & Int. ruled and quadric  \\
\hline 5   & 2 & 12   &  91  &     &  &  4.238  & Int.  bicubic patch and plane \\
\hline 6   & 1 & 4  & 9   &   &  &  0.201  & \\
\hline 7   & 1 &  10 & 34   &   &  &  0.352  &  \\
\hline 8   & 2  & 19   & 61   &  \checkmark & \checkmark &   1.031 &     \\
\hline 9   &  2 &  9  &  55  &  \checkmark & \checkmark &   0.949 &  \\
\hline
\end{tabular}

\vspace{0.4 cm}

\smallskip
{\bf Table 2:} 3D Examples.
\end{center}

The pictures corresponding to these curves are shown in Figure \ref{3d}. Notice that the timing in Ex. 3 is considerably higher, which is expectable because both the Weierstrass curve and the mapping ${\bf x}(t,s)$ are dense and with high degree.

\begin{figure}[ht]
\begin{center}
\centerline{$\begin{array}{ccc}
\includegraphics[scale=0.25]{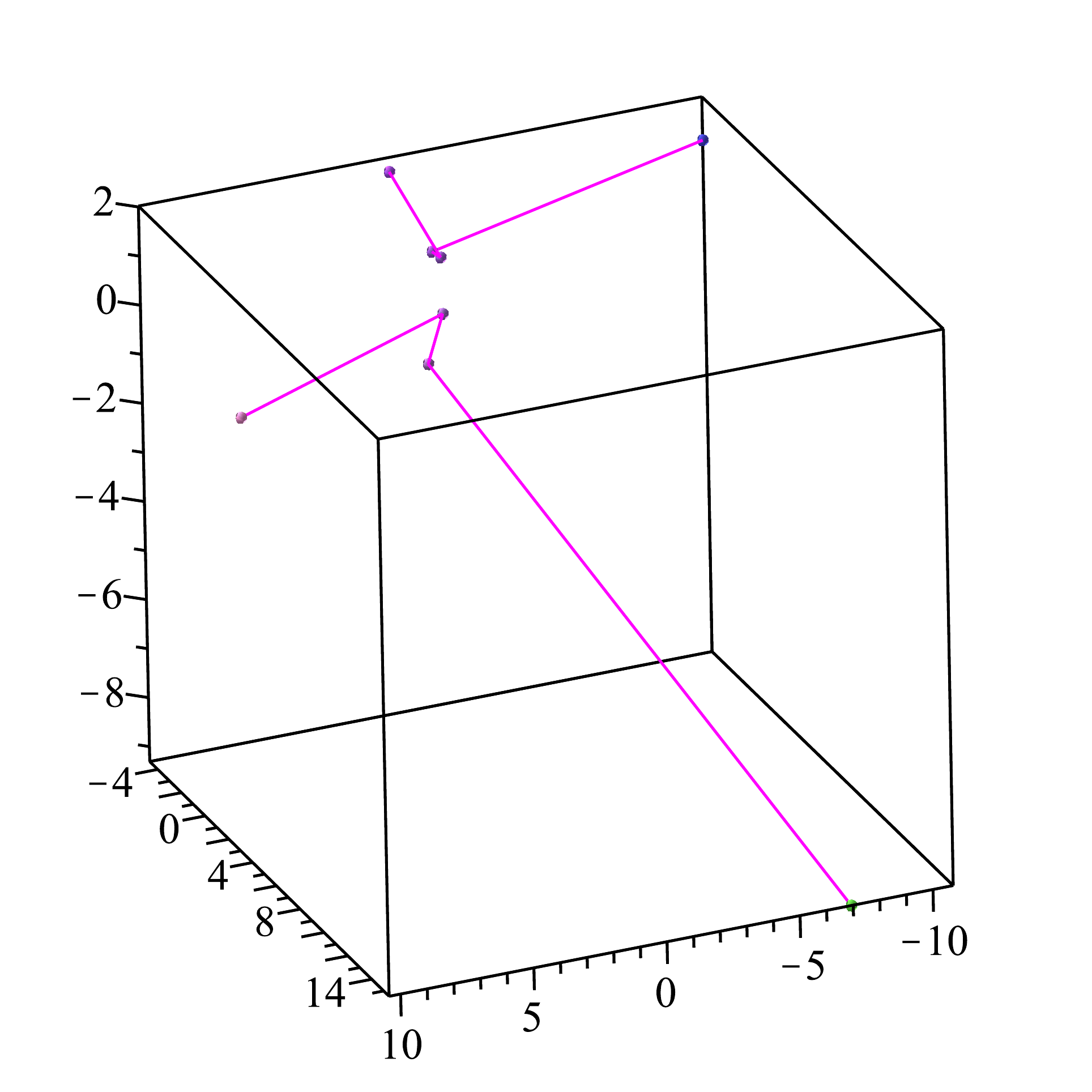} &
\includegraphics[scale=0.25]{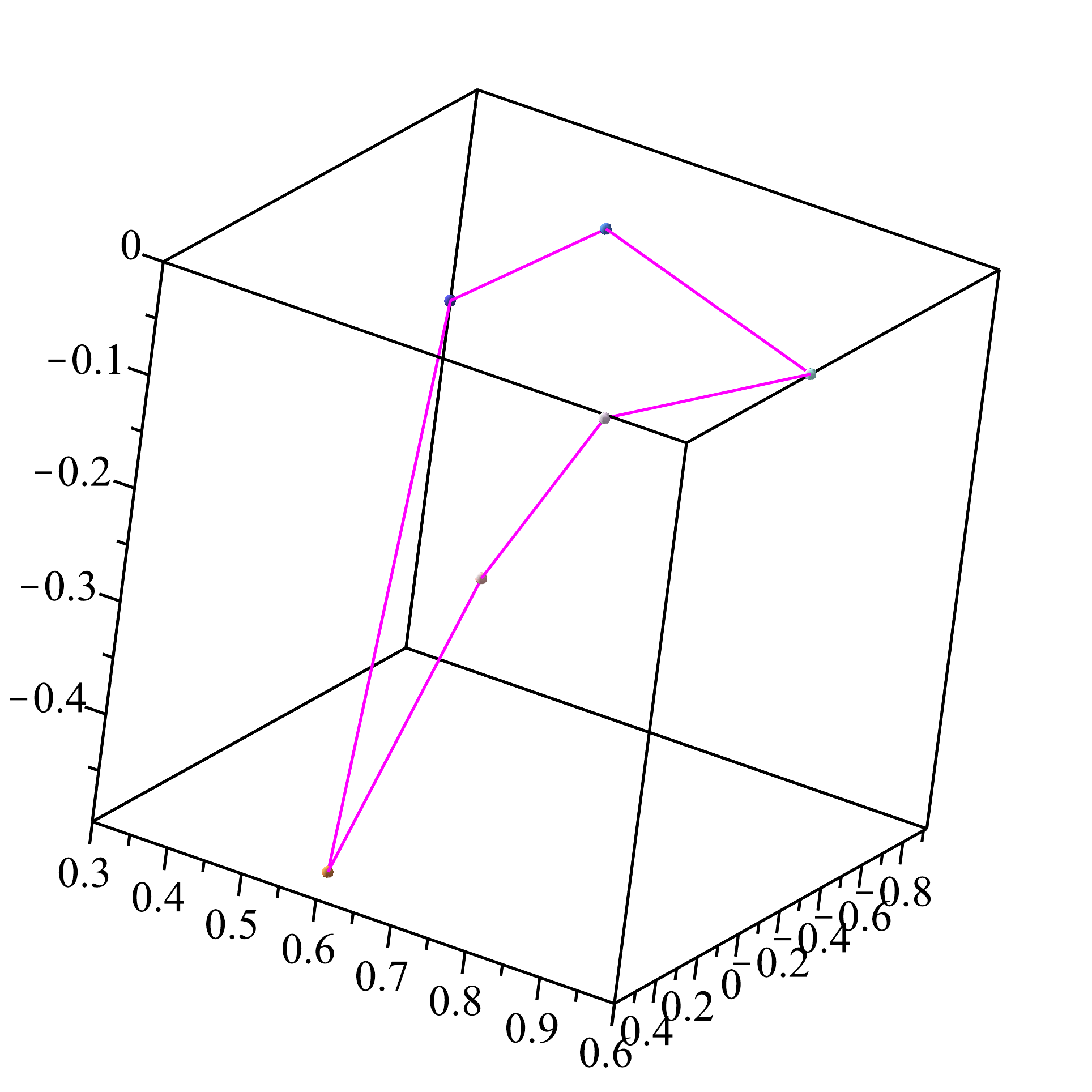} &
\includegraphics[scale=0.25]{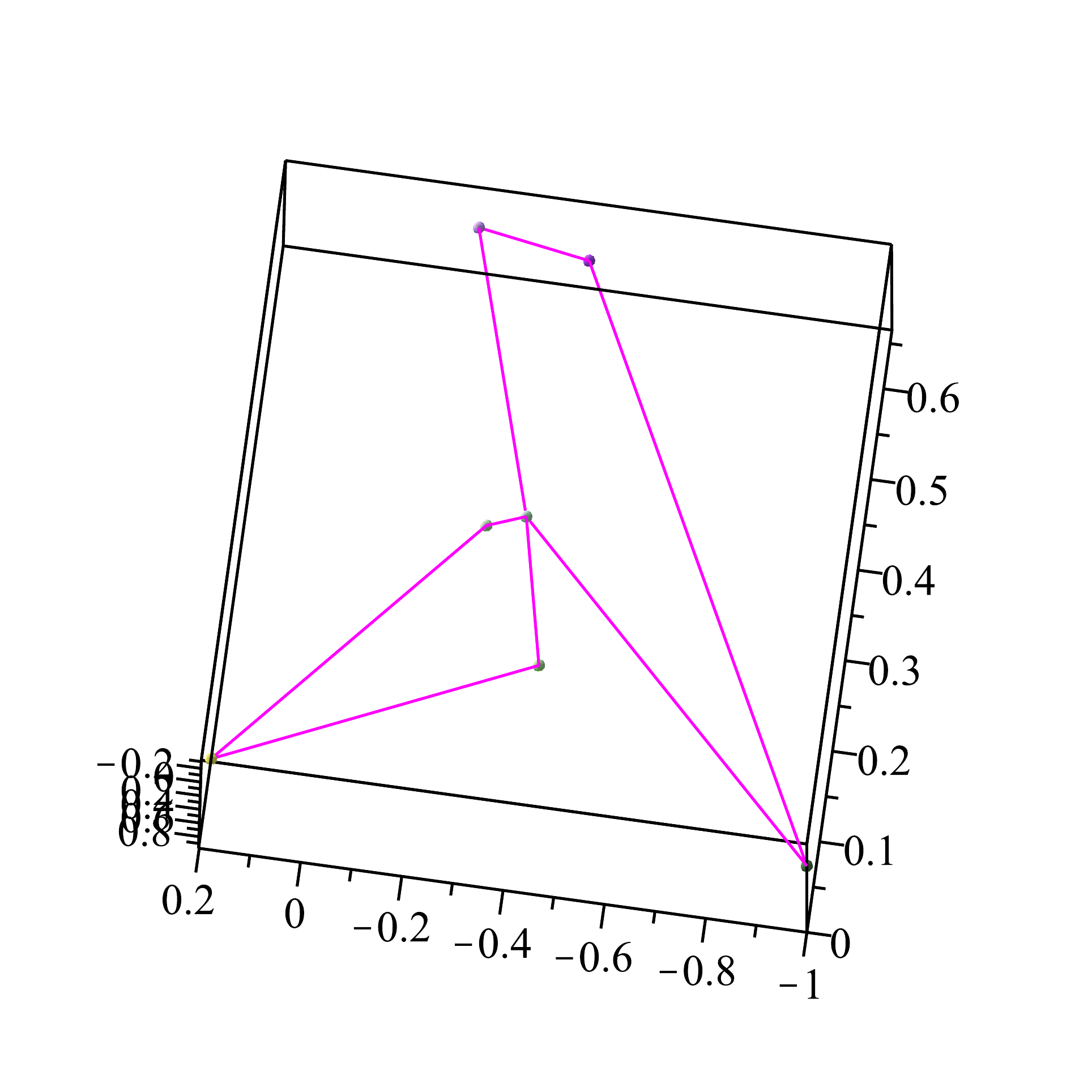} \\
\includegraphics[scale=0.25]{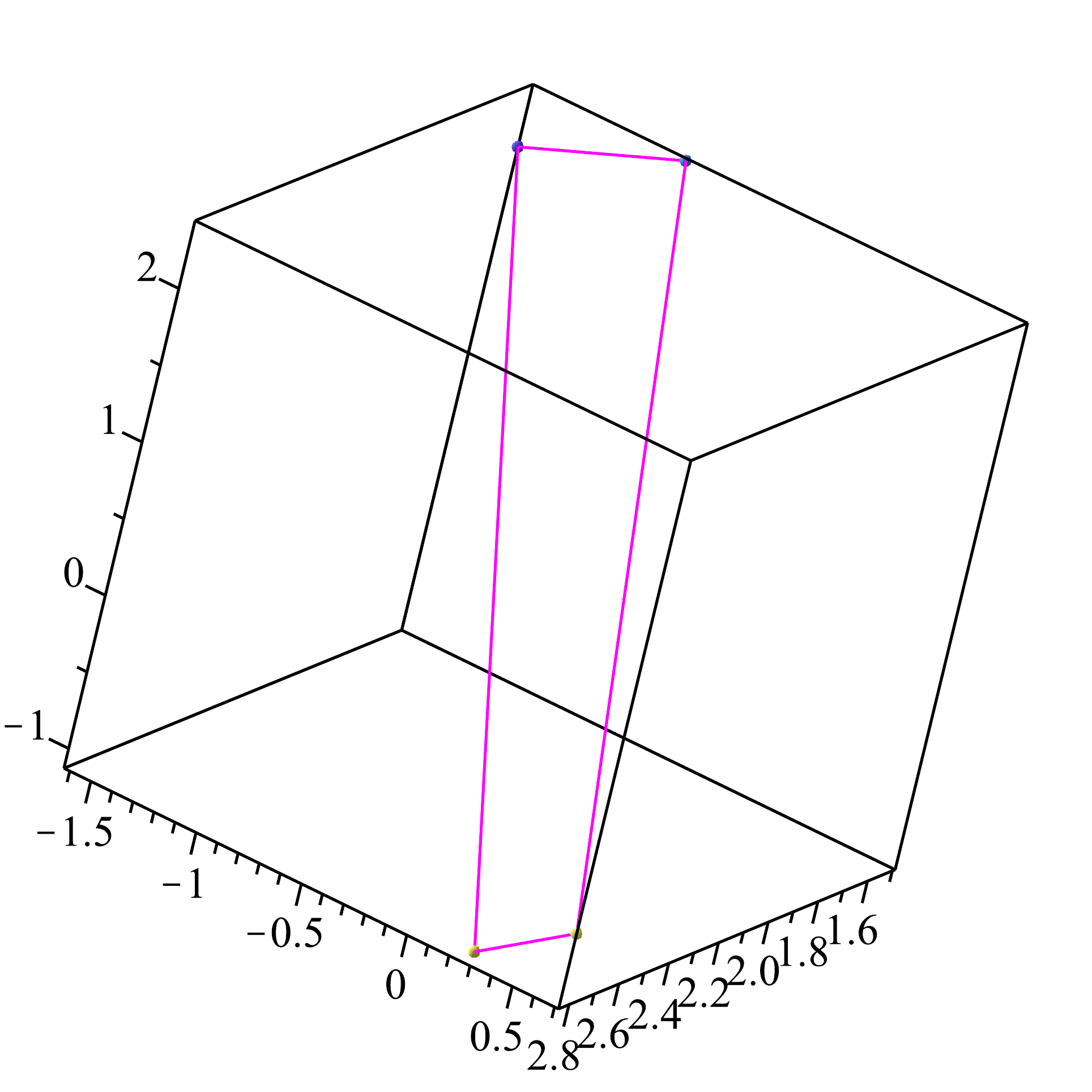} &
\includegraphics[scale=0.38]{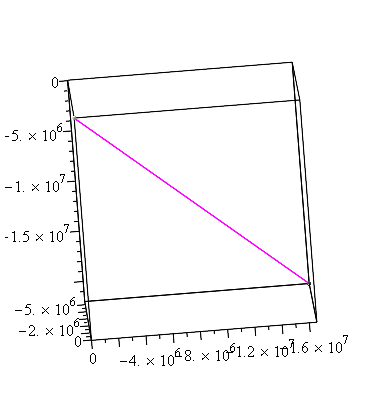} &
\includegraphics[scale=0.39]{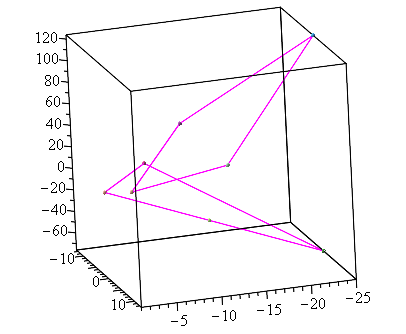}\\
\includegraphics[scale=0.39]{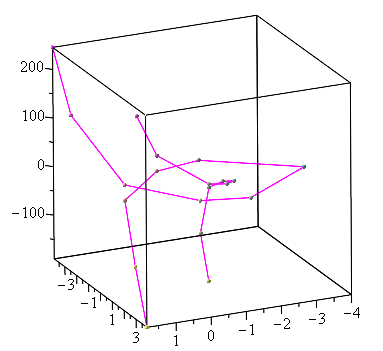} &
\includegraphics[scale=0.39]{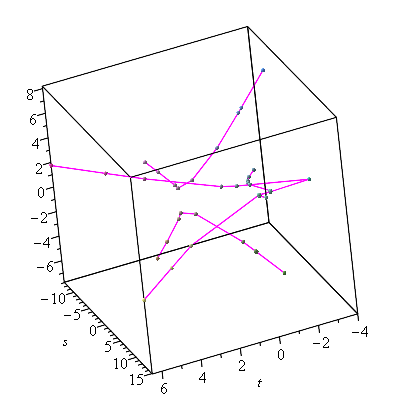} &
\includegraphics[scale=0.39]{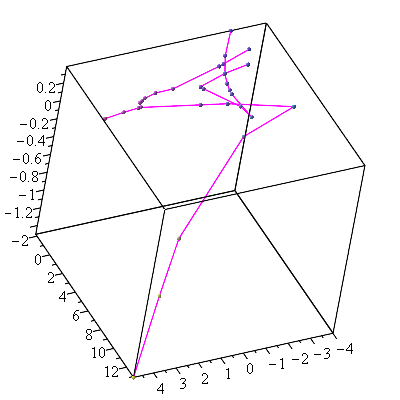} 
\end{array}$}
\end{center}
\caption{Examples of the 3D algorithm.}\label{3d}
\end{figure}

\section{Complexity and certification issues.} \label{sec-certification}

In this section we present the complexity of the algorithms presented
in the previous sections, and we elaborate on how to certificate the
topology of the curves. To certify the topology we must be sure
whether two different points $(t_i,s_i)\neq (t_j,s_j)$, both belonging
to ${\mathcal G}$, satisfy ${\bf x}(t_i,s_i)={\bf x}(t_j,s_j)$,
that is  whether they give rise to the same point $P\in {\mathcal C}$. We
first analyze the complexity of the algorithm without the
certification step: in particular, the timings corresponding to
Section \ref{sec-exp} do not include this certification. Then, we
address certification issues and provide the complexity of the
algorithm including the certification step. We analyze the algorithm
for 3D curves: the complexity bound is the same for 2D and 3D curves.

\subsection{Complexity (I)}

In this section we present the bit complexity analysis of the
algorithm without the certification step. This is the algorithm for which we perform experiments in Section~\ref{sec-exp}. We denote the maximum bitsize by
$\bsz(f)$ of the coefficients of a polynomial $f$. Additionally, we
denote by $\OO,\sO,\sOB$ the arithmetic complexity, the arithmetic
complexity neglecting logarithmic factors, and the bit complexity
(also neglecting logarithmic factors), respectively.

Let 
\[
  {\bf x}(t,s)= \left(
    \frac{a_{11}(t)+sa_{12}(t)}{b_{11}(t)+sb_{12}(t)},
    \frac{a_{21}(t)+sa_{22}(t)}{b_{21}(t)+sb_{22}(t)},
    \frac{a_{31}(t)+sa_{32}(t)}{b_{31}(t)+sb_{32}(t)}
  \right).
\]

We consider the following 3 polynomials: 
\[
\begin{array}{ccc}
X(t,s) &=& (b_{11}(t)+sb_{12}(t)) x - (a_{11}(t)+sa_{12}(t)),\\
Y(t,s) &=& (b_{21}(t)+sb_{22}(t)) y - (a_{21}(t)+sa_{22}(t)),\\
Z(t,s) &=& (b_{31}(t)+sb_{32}(t)) z - (a_{31}(t)+sa_{32}(t))).
\end{array}
\]
We also recall that $g(t, s) = s^2 - p(t)$. We assume that all the univariate polynomials in $t$, that is the
$a_{ij}(t),b_{ij}(t)$, and $p(t)$, have degree at most $d$, and that their
coefficients are integers of maximum bitsize at most $\tau$. 

The process of the algorithm goes as follows:

\noindent
\textit{(Step 1)} Compute the resultants
\[E_0 = \res_s(X, Y),\mbox{ }E_1 = \res_s(X, g).\]

The polynomial $E_0$ satisfies that $E_0 \in \ZZ[x, y, t]$. The degree of $E_0$ with
respect to $x$ and $y$ is 1 and with respect to $t$ is
$\leq 2d = \OO(d)$; moreover $\bsz(E_0) = \sO(\tau)$. The polynomial $E_1$ satisfies that $E_1 \in \ZZ[x, t]$. The degree of $E_1$ with respect
to $x$ is 2 and with respect to $t$ is $\leq 3d = \OO(d)$; also
$\bsz(E_1) = \sO(\tau)$.

Since the degree of $X, Y, Z$ and $g$ with respect to $x, y$, 
$s$ is at most 2, we can compute the resultants $E_0$ and $E_1$ by
performing a constant number of multiplications of univariate
polynomials in $t$. By recalling that the maximum degree with respect
to $t$ is $\sO(d)$, we deduce that the cost of computing $E_0$ and
$E_1$ is $\sOB(d \tau)$ \cite{MCA}.

\vspace{7pt}

\noindent
\textit{(Step 2)} Compute the subresultant sequence of $E_0$ and $E_1$ with
respect to $t$.

From the subresultant sequence we are interested in the polynomial of
degree 1 with respect to $t$.
This is the first subresultant polynomial; we can compute it in
$\sOB(d^4 \tau)$ \cite[Lemma~8]{DET}.  Let the coefficient of degree 1 of this
polynomial be ${\bf sres}_1 \in \ZZ[x, y]$ (i.e. the first principal subresultant). It has degree $\sO(d)$ and bitsize
$\sO(d \tau)$ \cite[Lemma~8]{DET}.

\vspace{7pt}

\noindent
\textit{(Step 3)}  Substitute the parametrization ${\bf x}(t,s)$ in ${\bf sres}_1$.

After clearing denominators we obtain a polynomial $M(t,s)\in \ZZ[t, s]$. The degree of $M(t,s)$ with respect to $t$ and $s$ is $\sO(d)$ and its
bitsize is $\sO(d^2\tau)$. This calculation of $M(t,s)$ involves $\OO(d)$ multiplications of
bivariate polynomials in $s$ and $t$. This cost is $\sOB(d^5 \tau)$
\cite{SRUR,MCA}.

\vspace{7pt}

\noindent
\textit{(Step 4)} Solve the polynomial system $M(t, s) = g(t, s) = 0$.

We can solve the system in $\sOB(d^7 \tau)$ (or $\sOB(d^8\tau)$)
\cite{ES-biv-12,BLMPRS-biv-16}.

After solving the system, we compute
the images under the birational mapping ${\bf x}(t,s)$ of all the
points $(t,s)$ computed along the way, and connect them properly.

The whole complexity is dominated by the complexity of solving the
polynomial system $(\Sigma) \{ M(t, s) = g(t, s) = 0 \}$, so we get a final bound of
$\sOB(d^7 \tau)$ (or $\sOB(d^8\tau)$), without including
certification.

\subsection{Certification and complexity (II)}\label{compl-2}

In this subsection we consider certification strategies, and we
present the complexity of the algorithm including this certification.
We perform the certification by exploiting the rational univariate
representation of the real roots of the polynomial system
$(\Sigma) \{ M(t, s) = g(t, s) = 0 \}$.

Within the complexity bound given in the previous subsection for
solving the bivariate system $(\Sigma)$, we can compute both an
isolating interval representation of the real roots, as a well a
(sparse) rational univariate representation (SRUR)
\cite{BLMPRS-biv-16}, see also \cite{SRUR}. The latter represents the
tuples $(t, s)$ of the solutions os $(\Sigma)$ as
$\left(\frac{F_1(\theta)}{F_0(\theta)},
  \frac{F_2(\theta)}{F_0(\theta)}\right)$, where $\theta$ runs over
all the (real) roots of a (univariate) polynomial $F(\theta)$ and
$F_0, F_1$, and $F_2$ are univariate polynomials.  This representation
involves univariate polynomials of degree $\sO(d^2)$ and bitsize
$\sO(d^3 \tau)$.

Now we want to identify which tuples of solutions of the polynomial
system $M(t, s) = g(t, s) = 0$ give rise to the same point on space
curve. Or in other words, we want to \textit{certify} when two tuples
give rise to the same point on the space curve.

Say that $(\alpha_1, \beta_1)$ and $(\alpha_2, \beta_2)$ are two
different solutions of the polynomial system $(\Sigma)$. Assume further that they
correspond to the roots $\theta_1$ and $\theta_2$ of the polynomial
$F(\theta)$.  Thus, their rational univariate representation is
\[
  \left(\frac{F_1(\theta_1)}{F_0(\theta_1)},
    \frac{F_2(\theta_1)}{F_0(\theta_1)}\right)
  \quad \text{ and } \quad
\left(\frac{F_1(\theta_2)}{F_0(\theta_2)},
  \frac{F_2(\theta_2)}{F_0(\theta_2)}\right) ,
\]
with
$F(\theta_1)=0$, $F(\theta_2)= 0$.

We check if they correspond to the same point by exploiting the
parametrization ${\bf x}$.  For example, to test if they result in the same
$x$-coordinate, we should test whether or not
$$\dfrac{a_{11}(\alpha_1)+\beta_1 a_{12}(\alpha_1)}{b_{11}(\alpha_1)+\beta_1 b_{12}(\alpha_1)} =
\dfrac{a_{11}(\alpha_2)+\beta_2
  a_{12}(\alpha_2)}{b_{11}(\alpha_2)+\beta_2 b_{12}(\alpha_2)}. $$
Clearing denominators, we get $\widehat{G}(\alpha_1,\alpha_2)=0$. Now
if we substitute the rational univariate representation of the roots
and clear denominators, then we get a new bivariate polynomial $G$,
and we need to test whether or not $G(\theta_1, \theta_2)=0$.

The degree of $G$ is $\sO(d^3)$, in $\theta_1$ and $\theta_2$
and its bitsize is $\sO(d^4 \tau)$. The complexity of computing $G$ involves the multiplication of $\sO(d)$ univariate polynomials and is $\sOB(d^8 \tau)$. The cost of this bivariate sign evaluation is $\sOB(d^{15}\tau)$.

We must perform this bivariate sign evaluation for every pair
$(\theta_i, \theta_j)$ of roots of $F$, 
and test for all coordinates $(x, y, z)$.
There are $\sO(d^4)$ pairs of solutions to test
and the total cost is $\sOB(d^{19}\tau)$. This complexity bound of certification dominates the overall complexity of the algorithm. 

We have implemented the certification part and the timings we get are
in agreement with this complexity: although there can be examples
where the computing time is reasonable, in general the timings are
very high and further research needs to be done. It seems plausible
to improve the complexity of certification by exploiting more
carefully aggregate separation bounds for the real roots of polynomial
systems \cite{emt-dmm-10}. For example, we can apply this aggregation
when we perform the time consuming sign evaluation of $G$ over all the
roots of the polynomial $F$. There should be a gain of a factor $d^2$
with this approach.

However, the most promising direction is to use more advanced
(probabilistic) tests for checking equality of real algebraic numbers
\cite{Blomer-sosr-91}.  The reader might notice that we do not really
need the actual sign evaluation of $G$ at two real algebraic numbers.
What we really need is to test whether or not the evaluation of
$G(\theta_1, \theta_2)$ is zero or not.

\subsection{Comparison of complexities with implicit algorithms.}

A possibility to compute the topology of ${\mathcal C}$ is to compute first an implicit representation of the curve, and then to apply an algorithm to complete the topology of an implicit curve. In the planar case, the implicit representation requires just one bivariate polynomial $f(x,y)$, that can be computed using Gr\"obner bases. Denoting the degree of $f(x,y)$ by $n$, and denoting by $\tau_f$ the bitsize of the coefficients of $f$, the complexity of computing the topology of $f(x,y)=0$ is $\sO_B(n^6+n^5\tau_f)$. In our case $n=\sO(d)$ and $\tau_f=\sO(d\tau)$, so we reach a complexity of $\sO_B(d^6\tau)$, certainly better than the bound we give in Subsection \ref{compl-2}. 

In the space case, however, the situation is much more difficult.
An implicit representation of ${\mathcal C}$ requires to
compute a basis for the ideal of the curve, which might have more than
two polynomials. Even if ${\mathcal C}$ is implicitly defined by only
two polynomials $f_i(x,y,z)$, with $i=1,2$, the known complexities for
implicit algorithms are worse than ours. In \cite{Diatta}, one has the
bound $\sO(n^{21}\tau_f)$, where $n,\tau_f$ are bounds for the degrees
and bitsizes of the $f_i$, respectively. For the same case, in
\cite{Lazard13} one has the bound $\sO(n^{37}\tau_f)$.

\section{Conclusion.}

We have presented algorithms to compute the topology of 2D and 3D hyperelliptic curves that do not require to compute or make use of the implicit representation of the curve. The 
main idea is to see the hyperelliptic curve as the image of a planar curve, the Weierstrass form of the curve, under a birational mapping of the plane or the space. Seeing the curve this way, the algorithms determines how the topology of the Weierstrass form changes when the birational mapping is applied. While a not completely certified algorithm produces good and fast results, a completely certified algorithm is much slower, although it is competitive in the space case, in terms of complexity, with algorithms using an implicit representation of the curve. Some lines of improvement to speed up the certification are suggested in the paper. We plan to exploit these ideas in the future to get a faster, certified, algorithm. 

\section*{References}


\begin{thebibliography}{10}

\bibitem{AS07}
J.~G. Alc\'azar and J.R. Sendra.
\newblock Local shape of offsets to algebraic curves.
\newblock {\em Journal of Symbolic Computation}, 42:338--351, 2007.

\bibitem{ADC15}
Juan~Gerardo Alc{\'a}zar, Jorge Caravantes, and Gema~M D{\'\i}az-Toca.
\newblock A new method to compute the singularities of offsets to rational
  plane curves.
\newblock {\em Journal of Computational and Applied Mathematics}, 290:385--402,
  2015.


\bibitem{ADT10}
Juan~Gerardo Alc{\'a}zar and Gema~Mar{\'\i}a D{\'\i}az-Toca.
\newblock Topology of 2d and 3d rational curves.
\newblock {\em Computer Aided Geometric Design}, 27(7):483--502, 2010.

\bibitem{AS05}
Juan~Gerardo Alc{\'a}zar and J~Rafael Sendra.
\newblock Computation of the topology of real algebraic space curves.
\newblock {\em Journal of Symbolic Computation}, 39(6):719--744, 2005.

\bibitem{EKBS}
Eric Berberich, Pavel Emeliyanenko, Alexander Kobel, and Michael Sagraloff.
\newblock Arrangement computation for planar algebraic curves.
\newblock In {\em Proceedings of the 2011 International Workshop on
  Symbolic-Numeric Computation}, pages 88--98. ACM, 2012.

\bibitem{Berb}
Eric Berberich, Pavel Emeliyanenko, Alexander Kobel, and Michael Sagraloff.
\newblock Exact symbolic--numeric computation of planar algebraic curves.
\newblock {\em Theoretical Computer Science}, 491:1--32, 2013.

\bibitem{Bizzarri}
Michal Bizzarri, Miroslav L{\'a}vi{\v{c}}ka, and Jan Vr{\v{s}}ek.
\newblock Piecewise rational approximation of square-root parameterizable
  curves using the weierstrass form.
\newblock {\em Computer Aided Geometric Design}, 56:52--66, 2017.

\bibitem{Blomer-sosr-91}
Johannes Blomer.
\newblock Computing sums of radicals in polynomial time.
\newblock In {\em Foundations of Computer Science, 1991. Proceedings., 32nd
  Annual Symposium on}, pages 670--677. IEEE, 1991.

\bibitem{BLMPRS-biv-16}
Yacine Bouzidi, Sylvain Lazard, Guillaume Moroz, Marc Pouget, Fabrice
  Rouillier, and Michael Sagraloff.
\newblock Solving bivariate systems using rational univariate representations.
\newblock {\em J. Complex.}, 37(C):34--75, December 2016.

\bibitem{Caravantes}
Jorge Caravantes, Gema~M D{\'\i}az-Toca, Laureano Gonz{\'a}lez-Vega, and Ioana
  Necula.
\newblock An algebraic framework for computing the topology of offsets to
  rational curves.
\newblock {\em Computer Aided Geometric Design}, 52:28--47, 2017.

\bibitem{Lazard13}
Jin-San Cheng, Kai Jina, and Daniel Lazard.
\newblock Certified rational parametric approximation of real algebraic space
  curves with local generic position method.
\newblock {\em Journal of Symbolic Computation}, 58:18--40, 2013.

\bibitem{CLPPRT09}
Jinsan Cheng, Sylvain Lazard, Luis Pe{\~n}aranda, Marc Pouget, Fabrice
  Rouillier, and Elias Tsigaridas.
\newblock On the topology of planar algebraic curves.
\newblock In {\em Proceedings of the twenty-fifth annual symposium on
  Computational geometry}, pages 361--370. ACM, 2009.

\bibitem{Daouda}
Diatta~Niang Daouda, Bernard Mourrain, and Olivier Ruatta.
\newblock On the computation of the topology of a non-reduced implicit space
  curve.
\newblock In {\em Proceedings of the twenty-first international symposium on
  Symbolic and algebraic computation}, pages 47--54. ACM, 2008.

\bibitem{Diatta}
Daouda Diatta.
\newblock {\em Calcul effectif de la topologie de courbes et surfaces
  algebriques reelles}.
\newblock Ph. D. Thesis, Universite de Limoges, 2009.

\bibitem{DET}
Dimitrios~I Diochnos, Ioannis~Z Emiris, and Elias~P Tsigaridas.
\newblock On the asymptotic and practical complexity of solving bivariate
  systems over the reals.
\newblock {\em Journal of Symbolic Computation}, 44(7):818--835, 2009.

\bibitem{EKW07}
Arno Eigenwillig, Michael Kerber, and Nicola Wolpert.
\newblock Fast and exact geometric analysis of real algebraic plane curves.
\newblock In {\em Proceedings of the 2007 international symposium on Symbolic
  and algebraic computation}, pages 151--158. ACM, 2007.

\bibitem{Kahoui}
Mohammed El~Kahoui.
\newblock Topology of real algebraic space curves.
\newblock {\em Journal of Symbolic Computation}, 43(4):235--258, 2008.

\bibitem{ES-biv-12}
Pavel Emeliyanenko and Michael Sagraloff.
\newblock On the complexity of solving a bivariate polynomial system.
\newblock In {\em Proceedings of the 37th International Symposium on Symbolic
  and Algebraic Computation}, pages 154--161. ACM, 2012.

\bibitem{emt-dmm-10}
Ioannis~Z Emiris, Bernard Mourrain, and Elias~P Tsigaridas.
\newblock The dmm bound: Multivariate (aggregate) separation bounds.
\newblock In {\em Proceedings of the 2010 International Symposium on Symbolic
  and Algebraic Computation}, pages 243--250. ACM, 2010.

\bibitem{GVN}
Laureano Gonz{\'a}lez-Vega and Ioana Necula.
\newblock Efficient topology determination of implicitly defined algebraic
  plane curves.
\newblock {\em Computer aided geometric design}, 19(9):719--743, 2002.

\bibitem{Hirsch}
Morris Hirsch.
\newblock {\em Differential Topology}.
\newblock Springer-Verlag, 1976.

\bibitem{Shafarevich}
Shafarevich I.R.
\newblock {\em Basic Algebraic Geometry 1 (Third edition)}.
\newblock Springer-Verlag, 2013.

\bibitem{Sagra15}
Alexander Kobel and Michael Sagraloff.
\newblock On the complexity of computing with planar algebraic curves.
\newblock {\em Journal of Complexity}, 31(2):206--236, 2015.

\bibitem{SRUR}
Angelos Mantzaflaris, {\'E}ric Schost, and Elias Tsigaridas.
\newblock Sparse rational univariate representation.
\newblock In {\em ISSAC 2017-International Symposium on Symbolic and Algebraic
  Computation}, page~8, 2017.

\bibitem{hyp}
Alfred Menezes, Robert Zuccherato, and Yi-Hong Wu.
\newblock {\em An elementary introduction to hyperelliptic curves}.
\newblock Reseach Report CORR 96-19, Faculty of Mathematics, University of
  Waterloo, 1996.

\bibitem{SSV17}
J~Rafael Sendra, David Sevilla, and Carlos Villarino.
\newblock Algebraic and algorithmic aspects of radical parametrizations.
\newblock {\em Computer Aided Geometric Design}, 55:1--14, 2017.

\bibitem{SWP}
Juan~Rafael Sendra, Franz Winkler, and Sonia P\'erez-D\'{\i}az.
\newblock {\em Rational Algebraic Curves: A Computer Algebra Approach}.
\newblock Springer Verlag, 2007.

\bibitem{gmdt}
G.M.~D\'iaz Toca.
\newblock \url{http://webs.um.es/gemadiaz/miwiki/doku.php?id=papers}, 2018.

\bibitem{MCA}
Joachim Von Zur~Gathen and J{\"u}rgen Gerhard.
\newblock {\em Modern computer algebra}.
\newblock Cambridge university press, 2013.

\bibitem{walker}
R.J. Walker.
\newblock {\em Algebraic Curves}.
\newblock Princeton University Press, 1950.

\end{thebibliography}

\newpage

\section{Appendix I: proofs.}

In this section we provide the proofs of some results in Section \ref{sec-planar}. We start with Proposition \ref{othersing}.

\medskip

\begin{proof} (of Proposition \ref{othersing}) Let ${\mathcal{V}}$ be the variety (the curve) in $\mathbb{R}^4(t,s,x,y)$ defined as 
\[
{\mathcal V}=V(\mbox{num}(x-x(t,s)),\mbox{num}(y-y(t,s)),g(t,s)),
\]
and let $\widehat{\mathcal V}=\Pi_{txy}({\mathcal V})$ be the projection of ${\mathcal V}$ onto $\mathbb{R}^3(t,x,y)$; notice that $\widehat{\mathcal V}\subset V(\xi_1,\xi_2)$. Suppose that $(t_0,s_0,x_0,y_0)$ is smooth in ${\mathcal V}$. Using the Jacobian matrix of $F_1(t,s,x)=\mbox{num}(x-x(t,s))$, $F_2(t,s,y)=\mbox{num}(y-y(t,s))$, $g(t,s)$ and condition \eqref{lacondi}, we observe that the tangent line to ${\mathcal V}$ at $(s_0,t_0,x_0,y_0)$ is parallel to $(-g_s(t_0,s_0),g_t(t_0,s_0),0,0)$. If $g_s(t_0,s_0)\neq 0$ (i.e. if $s_0\neq 0$) then the point $(t_0,x_0,y_0)$ is regular in $\widehat{\mathcal{V}}$ and the tangent line to $\widehat{\mathcal{V}}$ at $(t_0,x_0,y_0)$ is $\{x=x_0,y=y_0\}$, which is parallel to the $t$-axis. Therefore, $\xi_1(t,x_0)=0$ and $\xi_2(t,x_0,y_0)=0$ share the root $t_0$ with multiplicity higher than 1, and $\mathbf{sres}_1(x_0,y_0)=0$. If $g_s(t_0,s_0)=0$ (i.e. if $s_0=0$) then $(t_0,x_0,y_0)$ is singular in $\widehat{\mathcal V}$ and we can derive the same conclusion.

If, however, $(s_0,t_0,x_0,y_0)$ is a singular point of $\widehat{\mathcal{V}}$, then the tangent space to $\mathcal{V}$ at $(s_0,t_0,x_0,y_0)$, i.e. the kernel of the Jacobian matrix, consists of the vectors $(\alpha,\beta,0,0)$ with $\alpha,\beta\in {\Bbb C}$. Therefore, the line $\{x=x_0,y=y_0\}$ is tangent to $\widehat{\mathcal{V}}$ at $(t_0,x_0,y_0)$ and, therefore, all $\xi_i(t,x_0,y_0)$, $i=1,2$ have a multiple root at $t=t_0$. This implies that $\mathbf{sres}_1(x_0,y_0)=0$. 
\end{proof}

Now we prove Lemma \ref{t-factor}. From definitions of $\xi_1,\xi_2$ in Eq. \eqref{eq-xis} and taking into account that ${\bf x}$ can be written as in Eq. \eqref{ofx}, the polynomial $\xi_1(t,x)$ is the square-free part of the resultant with respect to $s$ of $g(t,s)=s^2-p(t) $ and
\begin{eqnarray*}
h(t,s,x)&:=&\mbox{num}(x-x(t,s))=\\&=&x (b_{11}(t)+sb_{12}(t))-(a_{11}(t)+sa_{12}(t))=\\&=&s(xb_{12}(t)-a_{12}(t))+xb_{11}(t)-a_{11}(t). 
\end{eqnarray*}

Since $\mathrm{degree}_s(g)=2$ and $\mathrm{degree}_s(h)\leq 1$, it is easy to compute such a resultant; if $\mathrm{degree}_s(h)=1$, i.e. if $x(t,s)$ explicitly depends on $s$, then
\begin{equation}\label{resxg}
\mbox{Res}_s(h,g)= \left(  b_{11}^2-p\,b_{12}^2  \right)x^2- 2\left(a_{11}\,b_{11}-p \,a_{12}\, b_{12}  \right) x +a_{11}^2- p\,a_{12}^2,
\end{equation}
where $b_{ij}=b_{ij}(t)$, $a_{ij}=a_{ij}(t)$ for $i=1,2$, $j=1,2$. If $\mathrm{degree}_s(h)=0$, i.e. if $x(t,s)$ does not depend on $s$, then 
\begin{equation}\label{resxg2}
\mbox{Res}_s(h,g)= h=x (b_{11}(t)+sb_{12}(t))-(a_{11}(t)+sa_{12}(t)).
\end{equation}

As for $\xi_2(t,x,y)$, that is is the square-free part of the resultant with respect to $s$ of $h(t,s,x)$ and 
\begin{eqnarray*}
j(t,s,y)&:=&\mbox{num}(y-y(t,s))=\\&=&y (b_{21}(t)+sb_{22}(t))-(a_{21}(t)+sa_{22}(t))= \\ &=&s(yb_{22}(t)-a_{22}(t))+yb_{21}(t)-a_{21}(t). 
\end{eqnarray*}
If $\mathrm{degree}_s(h)=\mathrm{degree}_s(j)=1$, i.e. if both $x(t,s)$ and $y(t,s)$ explicitly depend on $s$, then
\begin{equation}\label{resxy}
\mbox{Res}_s(h,j)=(a_{22}b_{11}-a_{21}b_{12})x+(a_{11}b_{22}-a_{12}b_{21})y+(b_{12}b_{21}-b_{11} b_{22} )xy-a_{11}a_{22}+a_{12}a_{21},
\end{equation}
where $b_{ij}=b_{ij}(t)$, $a_{ij}=a_{ij}(t)$ for $i=1,2$, $j=1,2$. If $\mathrm{degree}_s(h)=0$, i.e. if $x(t,s)$ does not depend on $s$, then 
\begin{equation}\label{resxy2}
\mbox{Res}_s(h,j)=h=x (b_{11}(t)+sb_{12}(t))-(a_{11}(t)+sa_{12}(t)),
\end{equation}
and if $\mathrm{degree}_s(j)=0$, i.e. if $y(t,s)$ does not depend on $s$, then 
\begin{equation}\label{resxy2}
\mbox{Res}_s(h,j)=j=yb_{21}(t)-a_{21}(t).
\end{equation}

\begin{proof} (of Lemma \ref{t-factor}) ``$\Leftarrow$" Suppose that $t_0$ corresponds to a base point. The resultant of $h(t,s,x)$ and $g(t,s)$ is equal to Equation (\ref{resxg}), and considered as a polynomial in $x$, it is easy to see that all its coefficients vanish at $t=t_0$. Thus, $t-t_0$ divides $\xi_1(x,t)$.
Likewise, the resultant of $h(t,s,x)$ and $j(t,s,y)$ is equal to  Equation (\ref{resxy}), and we can check that all its coefficients vanish in $t=t_0$. Thus, $t-t_0$ divides also $\xi_2(x,t)$.

\bigskip

``$\Rightarrow$" If $t-t_0$ divides $\xi_1$, then, by properties of resultants, since the leading coefficient of $g(t,s)$ with respect to $s$ is 1, there is $s_0$ with $g(t_0,s_0)=0$ and
\[
h(t_0,s_0,x)=x (b_{11}(t_0)+s_0b_{12}(t_0))-(a_{11}(t_0)+s_0a_{12}(t_0))=0; 
\] 
thus, $b_{11}(t_0)+s_0b_{12}(t_0)=a_{11}(t_0)+s_0a_{12}(t_0)=0$.

\medskip Next, if $t-t_0$ divides $\xi_2$, then either the leading coefficients of both $h(t,s,x)$ and $j(t,s,y)$ with respect to $s$ vanish at $t=t_0$, or there exists $s_1$ such that $h(t_0,s_1,x)=j(t_0,s_1,y)=0$ for all $x,y$. In the first case, we would have
$$b_{12}(t_0)=a_{12}(t_0)= b_{22}(t_0)=a_{22}(t_0)=0.$$
However, since also $b_{11}(t_0)+s_0b_{12}(t_0)=a_{11}(t_0)+s_0a_{12}(t_0)=0$, we should have 
$$
a_{11}(t_0)=a_{12}(t_0)=b_{12}(t_0)=b_{11}(t_0)=0,
$$
but this cannot happen because $\gcd(a_{11}, a_{12} , b_{11} , b_{12} )=1$. Therefore, there exists $s_1$ such that for all $x,y$
\[
h(t_0,s_1,x)=x (b_{11}(t_0)+s_1b_{12}(t_0))-(a_{11}(t_0)+s_1a_{12}(t_0))=0; 
\]
\[
j(t_0,s_1,y)= y (b_{21}(t_0)+s_1b_{22}(t_0))-(a_{21}(t_0)+s_1a_{22}(t_0)) =0.
\]
Then,
\[
\begin{array}{c}
b_{11}(t_0)+s_1b_{12}(t_0)=a_{11}(t_0)+s_1a_{12}(t_0)=0,\\
b_{21}(t_0)+s_1b_{22}(t_0)=a_{21}(t_0)+s_1a_{22}(t_0)=0.
\end{array}
\]

Since we also know that $b_{11}(t_0)+s_0b_{12}(t_0)=a_{11}(t_0)+s_0a_{12}(t_0)=0$, with $(t_0,s_0)\in {\mathcal G}$, we deduce that either $s_1=s_0$, or $b_{12}(t_0)=a_{12}(t_0)=0$. However, $b_{12}(t_0)=a_{12}(t_0)=0$ implies that $b_{11}(t_0)=a_{11}(t_0)=0$, which cannot happen because $\gcd(a_{11}, a_{12} , b_{11} , b_{12} )=1$. $s_0=s_1$ with $g(t_0,s_0)=0$. So, we can conclude that $t_0$ corresponds to a base point of ${\bf x}$.
\end{proof}
Finally, we prove Lemma \ref{tx-factor}. 

\begin{proof} (of Lemma \ref{tx-factor}) ``$\Leftarrow$" If $x(t,s)=x(t)$, then $\xi_1(t,x)=\xi_2(t,x,y)=b_{11}(t)x-a_{11}(t)$, and the result follows. 

``$\Rightarrow$" By way of contradiction, suppose that  $\xi_1(t,x)$ and $\xi_2(t,x,y)$ have a factor $\eta(t,x)$ depending on both $x,t$ and that both $x(t,s)$ and $y(t,s)$ depend on $s$. From Eq. \eqref{resxy2}, $y(t,s)$ explicitly depends on $s$. So suppose that $x(t,s)$ also depends on $s$. Then $\xi_2(t,x,y)$ is the square-free part of Eq. \eqref{resxy}, so $\eta(t,x)$ must be linear in $x$. Therefore either $\xi_2(t,x,y)$ coincides with $\eta(t,x)$, or $\xi_2(t,x,y)$ has another factor $\gamma(t,y)$ whose degree in $y$ is at most 1. Now we distinguish two cases:
\begin{itemize}
\item [(i)] If $\mathrm{degree}_y(\gamma)=1$, then for all $(t_0,y_0)$ such that $\gamma(t_0,y_0)=0$, either the leading coefficients of $h,j$ with respect to $s$ vanish at $(t_0,y_0)$ for all $x$, or there exists $s_0$ such that $h,j$ integrally vanish at $(t_0,s_0,y_0)$ for all $x$. The first possibility implies that both leading coefficients are zero modulo $\gamma(t,y)$, and this cannot happen because the leading coefficient of $h$ with respect to $s$ depends on $x$. But the second possibility cannot happen either, because that would imply that $x(t,s)$ has infinitely many base points. 
\item [(ii)] If $\mathrm{degree}_y(\gamma)=0$, then for all $(t_0,x_0)$ such that $\eta(t_0,x_0)=0$, either the leading coefficients of $h,j$ with respect to $s$ vanish at $(t_0,x_0)$ for all $y$, or there exists $s_0$ such that $h,j$ integrally vanish at $(t_0,s_0,y_0)$ for all $y$. Then we argue as before, this time with $j$ and $y(t,s)$. 
\end{itemize}
Thus we conclude that $x(t,s)$ cannot depend explicitly on $s$, and the result follows. 
\end{proof}

\newpage

\section{Appendix II: Parameterization of planar curves used in the experimentation}\label{appenII}

\underline{Example 1:} \medskip

$g(t)= s^2-(t^2+1);$
\medskip

${\bf x}(t,s)=\left(\dfrac{t^3+st+1}{t^4+2}, \dfrac{{t}^{3}-st+1}{ \left(t^4+2 \right)  \left( t-2 \right) }\right).$

\medskip \underline{Example 2:} \medskip
 
$g(t)= s^2-(t^2+9)(t-1);$
\medskip

${\bf x}(t,s)=\left({\dfrac {s{t}^{2}-{t}^{4}-6\,{t}^{2}-3\,s+3}{{t}^{2}+1}},{\dfrac {t
 \left( s{t}^{2}-3\,s+8 \right) }{4\,{t}^{3}-89\,{t}^{2}+77\,t+69}}
\right).$

\medskip \underline{Example 3:} \medskip
 
$g(t)= {s}^{2}-10\,{t}^{6}-24\,{t}^{5}+12\,{t}^{4}-92\,{t}^{3}+48\,{t}^{2}-6
\,t-10;$
\medskip

${\bf x}(t,s)=\left({\dfrac {1+{t}^{6}-15\,{t}^{4}+38\,{t}^{3}+ \left( 3-3\,s  \right) t}{
{t}^{6}-18\,{t}^{4}+2\,{t}^{3}-9\,{t}^{2}+1}},{\dfrac {2+2\,{t}^{6}+3\,
{t}^{5}+13\,{t}^{3}- \left( 3\,s+15 \right) {t}^{2}}{{t}^{6}-18\,{t}^
{4}+2\,{t}^{3}-9\,{t}^{2}+1}}
\right).$

\medskip \underline{Example 4:} \medskip
 
$g(t)=s^2+(t+1)(t-2)(t^2-25);$
\medskip

${\bf x}(t,s)=\left({\dfrac {{t}^{4}-{t}^{3}+{t}^{2}+5\,s-t}{{t}^{6}+1}},{\dfrac {{t}^{4}+{
t}^{3}-{t}^{2}-5\,s+t}{{t}^{6}+1}}\right).$

\medskip \underline{Example 5:} \medskip
 
$g(t)=s^2-(t-1)(t^2-9)(t^2+1);$
\medskip

${\bf x}(t,s)=\left({\dfrac {{t}^{4}-{t}^{3}+{t}^{2}+5\,s-t}{{t}^{6}+1}},{\dfrac {{t}^{4}+{
t}^{3}-{t}^{2}-5\,s+t}{{t}^{6}+1}}\right).$

\medskip \underline{Example 6:} \medskip
 
$g(t)=s^2-(t-1)(t^2+9)(t^2+1);$
\medskip

${\bf x}(t,s)=\left({\dfrac {s{t}^{2}-{t}^{4}-6\,{t}^{2}-3\,s+3}{ \left( {t}^{2}+1
 \right) s}},-{\dfrac {t \left( s{t}^{2}-3\,s-8 \right) }{ \left( {t}^{
2}+1 \right) s}}
\right).$

 \medskip\underline{Example 7:} \medskip
 
$g(t)=s^2-(t-1)(t^2+9)(t^2+1);$
\medskip

${\bf x}(t,s)=\left({\dfrac {{t}^{4}-{t}^{3}-{t}^{2}+5\,s-t}{{t}^{6}+1}},{\dfrac {{t}^{5}+{
t}^{3}-{t}^{2}-5\,s+t}{{t}^{6}+1}}\right).$

\medskip \underline{Example 8:} \medskip
 
$g(t)= s^2-5t^4+16t^3-34t^2+16t-29s;$
\medskip

${\bf x}(t,s)=\left({\dfrac {2\,s{t}^{2}-3\,{t}^{4}-18\,{t}^{2}-2\,s-15}{ -2\left( {t
}^{2}+1 \right) s}},{\dfrac {6\,{t}^{3}-3\,{t}^{4}+2\,st-6\,{t}^{2}+6
\,t-3}{ \left( {t}^{2}+1 \right) s}}
\right).$

\medskip \underline{Example 9:} \medskip
 
$g(t)={s}^{2}+17\,{t}^{5}+75\,{t}^{4}+10\,{t}^{3}+7\,{t}^{2}+40\,t-42;$
\medskip

${\bf x}(t,s)=\left(\,{\dfrac {2\,s{t}^{2}-3\,{t}^{4}-18\,{t}^{2}-2\,s-15}{ -2\left( {t
}^{2}+1 \right) s}},{\dfrac {6\,{t}^{3}-3\,{t}^{4}+2\,st-6\,{t}^{2}+6
\,t-3}{ \left( {t}^{2}+1 \right) s}}\right).$

\section{Appendix III: Parameterizations of the space curves used in the experimentation}\label{appenIII}

\medskip \underline{Example 1:} \medskip
 \begin{eqnarray*}
g(t,s)&=&
{s}^{2}-14\,{t}^{10}+32\,{t}^{9}-44\,{t}^{8}+36\,{t}^{7}-13\,{t}^{6}-
28\,{t}^{5}+48\,{t}^{4}\\&&-36\,{t}^{3}-8\,{t}^{2}+24\,t-12;
\end{eqnarray*}

${\bf x}(t,s)=\left( \dfrac {s+8t^5-10t^4+9t^3-4t+6}{6},\dfrac {s-t^5+2t^4-3t^3+2t-3}{3}, \dfrac {4t^5-2t^4-3t^3-s+4t}{6}
\right)$

\medskip \underline{Example 2:} \medskip
 
$g(t,s)=   s^2-2t^5+8;$
\medskip

${\bf x}(t,s)=\left({\dfrac {s{t}^{3}-{t}^{5}+{t}^{4}+8}{{t}^{6}+{t}^{4}+8}},{\dfrac {{t}^
{6}-{t}^{5}+s{t}^{2}+8}{{t}^{6}+{t}^{4}+8}},{\dfrac {2s-2\,{t}^{3}-2\,{t}
^{2}}{{t}^{6}+{t}^{4}+8}}
 \right).$

\medskip \underline{Example 3:} 
 \begin{eqnarray*}
g(t,s)&=& {s}^{2}+3\,{t}^{16}+6\,{t}^{14}+4\,{t}^{13}+3\,{t}^{12}+4\,{t}^{11}+{t
}^{10}-2\,{t}^{9} -6\,{t}^{8}\\ &&4\,{t}^{7}-2\,{t}^{6}+6\,{t}^{5}+4\,{t}^{
4}+2\,{t}^{3}-1;
\end{eqnarray*}
 \begin{eqnarray*}{\bf x}(t,s)&=&\left( {\dfrac {{t}^{7}-{t}^{9}-{t}^{8}+{t}^{6}+2\,{t}^{5}+ \left( s+2
 \right) {t}^{3}-3\,{t}^{2}+2\,t-s-1}{3\,{t}^{8}+3\,{t}^{6}-2\,{t}^{5}
+{t}^{4}-4\,{t}^{3}+3\,{t}^{2}-2\,t+2}}\right.,\\&&{\dfrac {3\,{t}^{11}+3\,{t}^{9}
+{t}^{8}-{t}^{6}+{t}^{4}-{t}^{3}+2\,{t}^{2}+ \left( s-1 \right) t-s+1
}{3\,{t}^{8}+3\,{t}^{6}-2\,{t}^{5}+{t}^{4}-4\,{t}^{3}+3\,{t}^{2}-2\,t+
2}},\\&&\left.{\dfrac{{t}^{2} \left( -{t}^{7}+2\,{t}^{6}-2\,{t}^{5}-{t}^{3}+{t}^
{2}+2+ \left( s-1 \right) t \right) }{3\,{t}^{8}+3\,{t}^{6}-2\,{t}^{5}
+{t}^{4}-4\,{t}^{3}+3\,{t}^{2}-2\,t+2}}
\right).\end{eqnarray*}

\medskip \underline{Example 4:} \medskip
 
$g(t,s)= {s}^{2}+{t}^{8}-6\,{t}^{7}+5\,{t}^{6}-16\,{t}^{5}-6\,{t}^{4}+2\,{t}^{3
}+103\,{t}^{2}-24\,t+6;
$
%
 \begin{eqnarray*}\hspace*{-2cm}{\bf x}(t,s)&=&\left({\dfrac {{t}^{5}+3\,{t}^{4}+4\,{t}^{3}+23\,{t}^{2}+s-t+2}{t \left( {t}
^{4}+6\,{t}^{2}+2 \right) }}\right.,
\\&&{\dfrac {-2\,{t}^{5}+6\,{t}^{4}-5\,{t}^{3}
+ \left( s+7 \right) {t}^{2}-3\,t+s-2}{t \left( {t}^{4}+6\,{t}^{2}+2
 \right) }},
 \\&&\left.{\dfrac {{t}^{6}-3\,{t}^{5}+3\,{t}^{4}-8\,{t}^{3}+2\,{t}^{2
}+ \left( 2\,s-14 \right) t+2}{t \left( {t}^{4}+6\,{t}^{2}+2 \right) }
} \right).\end{eqnarray*}

\underline{Example 5:} \medskip
 
$g(t,s)= {s}^{2}+3\,{t}^{6}+522\,{t}^{5}-1053\,{t}^{4}+1648\,{t}^{3}-588\,{t}^{
2}+192\,t+76;$

{\scriptsize  \begin{eqnarray*}\hspace*{-2cm}{\bf x}(t,s)&=&\left(\dfrac {54{t}^{9}+540{t}^{8}+378{t}^{7}-4110{t}^{6}+7713
{t}^{5}-4986{t}^{4}+ \left( 18s+1612 \right) {t}^{3}+ \left( 45
s+489 \right) {t}^{2}- \left( 27s+396 \right) t+15s+148}{
2 \left( 3{t}^{3}+18{t}^{2}-9t+7 \right) ^{2}},\right.\\
 &&
\dfrac {18
{t}^{9}+216{t}^{8}+594{t}^{7}+201{t}^{6}+1260{t}^{5}-243{t}^
{4}+ \left( 27s-538 \right) {t}^{3}+ \left( -27s+288 \right) {t}^{
2}+6t-18s-60}{2 \left( 3{t}^{3}+18{t}^{2}-9t+7 \right) ^{2}}
,\\
 &&
\left.{\dfrac {72{t}^{9}+756{t}^{8}+972{t}^{7}-3891{t}^{6}+9189
{t}^{5}-4689{t}^{4}+ \left( 45s+510 \right) {t}^{3}+ \left( 18
s+1443 \right) {t}^{2}- \left( 27s+642 \right) t-3s+186}{ 2\left( 
3{t}^{3}+18{t}^{2}-9t+7 \right) ^{2}}}\right).
\end{eqnarray*}}

\medskip \underline{Example 6:} \medskip
 
$g(t,s)= s^2+(t+1)(t-2)(t^2-25);$
\medskip

${\bf x}(t,s)=\left( -t^2 , -s ,t^3+t^2-5t-1 \right).$

\medskip \underline{Example 7:} \medskip
 
$g(t,s)= s^2-(t^2+9)(t^2+1);$
\medskip

${\bf x}(t,s)=\left( {\dfrac {s{t}^{2}-{t}^{4}-6\,{t}^{2}-3\,s+3}{ \left( {t}^{2}+1
 \right) s}},-{\dfrac {t \left( s{t}^{2}-3\,s+8 \right) }{ \left( {t}^{
2}+1 \right) s}},2\,{t}^{3}-s
\right).$

\medskip \underline{Example 8:} \medskip
 
$g(t,s)={s}^{2} -{t}^{6}+9\,{t}^{4}+{t}^{2}-9;$

\medskip

${\bf x}(t,s)=\left({\dfrac {-t^4+2\,s{t}^{2}-6\,{t}^{2}-3\,s+3}{ \left( t^2+1
 \right) s}},-{\dfrac {st^2-3\,s+8\,t}{ \left( t^2+1 \right) s}
},{\dfrac {{t}^{5}}{ \left( t^2+1 \right) s}}\right).$

\underline{Example 9:}\medskip

 $g(t,s)={s}^{2}-{t}^{5}+{t}^{4}+8\,{t}^{3}-8\,{t}^{2}+9\,t-9$

\medskip

${\bf x}(t,s)=\left({\dfrac {-{t}^{4}+s{t}^{2}-6\,{t}^{2}-3\,s+3}{ \left( {t}^{2}+1
 \right) s}},{\dfrac {3\,s-s{t}^{2}-8\,t}{ \left( {t}^{2}+1 \right) s}
},{\dfrac {{t}^{2}}{ \left( {t}^{2}+1 \right) s}}\right).$

\end{document}